\documentclass[twoside,11pt]{article}
\usepackage[preprint]{jmlr2e}
\usepackage{amsmath}
\usepackage{bm}
\usepackage{caption}
\usepackage{subcaption}
\usepackage{algorithm,setspace}
\usepackage[noend]{algpseudocode}
\usepackage{listings}
\usepackage[disable,textwidth=2cm,textsize=footnotesize]{todonotes}

\usepackage{tikz}
\usepackage{pgfplots}
\usepgfplotslibrary{groupplots}
\tikzset{subcap/.style={below,text width=6cm,inner ysep=0em}}

\pgfplotsset{compat=1.17}
\pgfplotsset{
  every axis/.append style={
    scale only axis,
    xlabel near ticks, ylabel near ticks,
    every axis title shift=3pt,
    cycle list={}, bar cycle list={},
    legend cell align=left,
    enlargelimits=0.05,
  }
  every axis plot/.append style={
    every mark/.append style={solid},
  }k}
  
\newif\ifloadfigs
\newcommand{\minimize}{\mathop{\mbox{minimize}{}}}

\newcommand{\E}{\mathbf{E}}

\newcommand{\Var}{\textup{Var}}
\newcommand{\R}{\mathbf{R}}

\newcommand{\ones}{\mathbf{1}}
\newcommand{\sign}{\textup{sign}}

\newcommand{\Tr}{\textup{Tr}}

\newcommand{\SNR}{\mathrm{SNR}}
\newcommand{\norm}[1] {\left \| #1 \right \|}

\newcommand{\cov}{\textup{Cov}}
\newcommand{\corr}{\textup{Corr}}
\newcommand{\var}{\textup{Var}}
\newcommand{\cF}{\mathcal{F}}
\newcommand{\cO}{\mathcal{O}}

\newtheorem{assumption}{Assumption}

\usepackage{lastpage}
\firstpageno{1}

\begin{document}

\title{Stochastic Approximation with Block Coordinate\\ Optimal Stepsizes}

\author{\name Tao Jiang \email taojiang@meta.com \\
       \name Lin Xiao \email linx@meta.com \\
       \addr Fundamental AI Research (FAIR)\\
       Meta Superintelligence Labs\\ 
       Seattle, Washington 98109, USA.
}
\editor{}

\maketitle

\begin{abstract}
We consider stochastic approximation with block-coordinate stepsizes and propose adaptive stepsize rules that aim to minimize the expected distance from the next iterate to an (unknown) target point.
These stepsize rules employ online estimates of the second moment of the search direction along each block coordinate. 
The popular Adam algorithm can be interpreted as a variant with a specific estimator.
By leveraging a simple conditional estimator, we derive a new method that obtains competitive performance against Adam but requires less memory and fewer hyper-parameters.
We prove that this family of methods converges almost surely to a small neighborhood of the target point, and the radius of the neighborhood depends on the bias and variance of the second-moment estimator. 
Our analysis relies on a simple aiming condition that assumes neither convexity nor smoothness, thus has broad applicability.
% Maybe should state "optimal point" as "target point"? Optimal with respect to what?
\end{abstract}

\begin{keywords}
stochastic approximation, block coordinate methods, conditional estimator, aiming condition, almost sure convergence. 
% \todo{TJ: Conditional estimator?}
\end{keywords}

\section{Introduction}
\label{sec:intro}
We consider unconstrained stochastic optimization problems of the form
\begin{equation}\label{eqn:stoch-opt}
\minimize_{x\in\R^n}~F(x):= \E_{\xi}[f(x,\xi)],
\end{equation}
where $x\in\R^n$ is the decision variable, $\xi$ is a random variable, and~$f$ is the loss function. 
In the context of statistical machine learning \citep[e.g.,][]{hastie2009elements}, $x$ represents the parameters of a prediction model, $\xi$~represents randomly sampled data from some (unknown) probability distribution, and $f(x,\xi)$ is the error in making predictions about~$\xi$ using the model parametrized by~$x$.
We assume that~$F$ is bounded below.
%and there exists an optimal solution~$x_*$.
%\todo{In this paper, $x_*$ is not really the minimizer of~$F$, so we do not name it here!}

Suppose that for any pair $x$ and $\xi$, we can evaluate the gradient of~$f$ with respect to~$x$, denoted as $\nabla f(x,\xi)$.
Starting with an initial point $x_0\in\R^n$, 
the classical stochastic approximation method \citep{RobbinsMonro51} generates a sequence $\{x_1,x_2,\ldots\}$ with the update rule
\begin{equation}\label{eqn:sgd}
x_{t+1} = x_t - \alpha_t  \nabla f(x_t,\xi_t),
\end{equation}
where $\alpha_t$ is the stepsize (often called the \emph{learning rate} in the machine learning literature). 
The convergence properties of this method are well studied
in the stochastic approximation literature
\citep[e.g.,][]{RobbinsMonro51,Wasan69book,KushnerYin03book}.
In particular, the stepsizes $\alpha_t$ should be nonnegative and unsummable, i.e., $\sum_{t=0}^\infty\alpha_t=\infty$.
Stronger guarantees such as almost sure convergence also require $\alpha_t$ to be square-summable, i.e., $\sum_{t=0}^\infty\alpha_t^2<\infty$
\citep[e.g.,][]{RobbinsMonro51,Blum1954multidim,chung1954,ROBBINS1971}.

Despite the rich literature on their convergence theory, stochastic approximation methods in practice often require heuristics and trial and error in choosing the stepsize sequence~$\{\alpha_t\}$.
Adaptive rules that can adjust stepsizes on the fly have been developed in both the optimization literature
\citep[e.g.,][]{Kesten58, MirzoakhmedovUryasev83, RuszczynskiSyski83stats, RuszczynskiSyski84IFAC, DelyonJuditsky93} 
and by the machine learning community
\citep[e.g.,][]{Jacobs88DBD, Sutton92IDBD, Schraudolph1999, MahmoodSutton12TuningFree}.
More recently, adaptive algorithms that use \emph{coordinate-wise} stepsizes have become very popular following the seminal works of AdaGrad \citep{DuchiHazanSinger2011adagrad} and Adam \citep{KingmaBa2014adam}, and they demonstrate significantly better performance in training large-scale deep neural networks. 

In this paper, we propose a family of stochastic approximation methods with block-coordinate stepsizes and present a general framework for their convergence analysis.

\subsection{Stochastic approximation with block-coordinate stepsizes}

We focus on stochastic approximation methods of the form
%with \emph{block-coordinate stepsizes}, specifically of the form
\begin{equation}\label{eqn:bc-sgd}
x_{t+1} = x_t - s_t \odot d_t ,
\end{equation}
where $d_t\in\R^n$ is a stochastic search direction, 
$s_t\in\R^n$ is a vector of coordinate-wise stepsizes, 
and $\odot$ denotes element-wise product (Hadamard product) of two vectors.
Other than using a different stepsize for each single coordinate, a block structure can be imposed where coordinates within each block share a common stepsize (see Section~\ref{sec:notations}).

The two most common choices for the search direction are:
the \emph{stochastic gradient}, i.e., $d_t=\nabla f(x_t, \xi_t)$,
and its \emph{exponential moving average (EMA)}. 
Letting $g_t=\nabla f(x_t,\xi_t)$, the EMA of stochastic gradient can be expressed as  
\begin{equation}\label{eqn:d-momentum}
d_t=\beta_1 d_{t-1} + (1-\beta_1) g_t,
\end{equation}
where $\beta_1\in[0,1)$ is a smoothing factor.
This is often called \emph{stochastic momentum}. 

The Adam algorithm \citep{KingmaBa2014adam} uses the stochastic momentum in~\eqref{eqn:d-momentum} as its search direction and sets the coordinate-wise stepsizes as
\begin{equation}\label{eqn:adam-stepsize}
s_{t,i}=\frac{\alpha_t}{\sqrt{v_{t,i}}+\epsilon}, \qquad i=1,\ldots,n,
\end{equation}
where $\alpha_t\in\R$ is a common stepsize \emph{schedule} for all coordinates, $v_{t,i}$ is the EMA of $g_{t,i}^2$, and $\epsilon>0$ is a small constant to improve numerical stability when $v_{t,i}$ becomes very close to zero.
Here, computing $v_{t,i}$ requires a second smoothing factor~$\beta_2\in[0,1]$, i.e., 
\begin{equation}\label{eqn:adam-v}
v_{t,i}=\beta_2 v_{t-1,i} + (1-\beta_2)g^2_{t,i}, \qquad i=1,\ldots,n.
\end{equation}
For Adam to perform well, the second smoothing factor $\beta_2$ is often set to be larger than~$\beta_1$. For example, a pair of commonly used choices is $\beta_1=0.9$ and $\beta_2=0.99$.

Adam \citep{KingmaBa2014adam} and its variant AdamW \citep{Loshchilov2019AdamW}, which implements \emph{decoupled weight decay}, have been very successful in training large-scale deep neural networks.
However, theoretical understanding of their superior performance is still incomplete despite many recent efforts
\citep[e.g.,][]{bock2018improvement,balles2018dissecting,reddi2019convergence,zou2019sufficient,defossez2020simple,Luo2022neurips,kunstner2024heavy}.
On the other hand, there are also many works that propose new variants or alternatives to Adam and AdamW, either starting from fundamental optimization principles \citep[e.g.,][]{yao2021adahessian,gomes2024adafisher,hwang2024fadam,jordan2024muon,lin2024can,pethick2025LMOs}
or based on empirical algorithmic search \citep[e.g.,][]{chen2023symbolic,zhang2024adam}.
But all have limited success. 
Adam and especially AdamW are still the most widely used algorithms for training large deep learning models, and their effectiveness remains a myth.

\subsection{Contributions and outline}

In this paper, we propose a family of \emph{block-coordinate optimal stepsize} (BCOS) rules for stochastic approximation.
BCOS gives a novel interpretation of Adam and AdamW and provides their convergence analysis as special cases of a general framework. 
Moreover, we derive variants of BCOS that obtain competitive performance against Adam(W) but require less memory (optimizer states) and fewer hyper-parameters. 
More specifically:

\begin{itemize}
\item 
In Section~\ref{sec:derivation}, we derive BCOS by minimizing the expected distance from the next iterate to a target point.
While the optimal stepsizes cannot be computed exactly, we make several simplifications and approximate the second moments of the search direction with simple EMA estimators.
\item
In Section~\ref{sec:instantiation}, we instantiate BCOS with specific search directions. 
In particular, we show that RMSProp \citep{rmsprop2012lecture} and Adam \citep{KingmaBa2014adam} can be interpreted as special cases of BCOS.
With the stochastic momentum as search direction, we leverage a simple conditional estimator to derive a variant that requires less memory and fewer hyper-parameters than Adam.
Integrating with decoupled weight decay \citep{Loshchilov2019AdamW} gives the BCOSW variants that enjoy competitive performance as AdamW.
\item
In Section~\ref{sec:experiments}, we present numerical experiments to compare several variants of BCOS(W) with Adam(W) on several deep learning tasks.
We observe that BCOSW performs better than BCOS, and BCOSW with the conditional estimator achieves top performance with fewer optimizer states and less hyper-parameter tuning.
\item 
In Section~\ref{sec:convergence}, we present the convergence analysis of BCOS(W) based on a simple aiming condition, which assumes neither convexity nor smoothness.
We prove that the BCOS family of methods (including Adam and AdamW) converge almost surely to a small neighborhood of a target point, and the radius of the neighborhood depends on the bias and variance of the second-moment estimator. 
\item 
We conclude the paper in Section~\ref{sec:conclusion}.
\end{itemize}

\paragraph{Generality}
The general form of stochastic approximation 
\citep[e.g.,][]{RobbinsMonro51,Wasan69book,KushnerYin03book}
concerns solving a nonlinear system of equations $G(x)=0$ where the map $G:\R^n\to\R^n$ only admits noisy evaluations in the form of $g(\cdot,\xi)$, where $\xi$ is a random variable, under the assumption that $G(x)=\E_\xi[g(x,\xi)]$.
In this case, stochastic approximation takes the form
\begin{equation}\label{eqn:general-sa}
x_{t+1} = x_t - \alpha_t\, g(x_t, \xi_t).
\end{equation}
Apparently, the stochastic gradient method~\eqref{eqn:sgd} is a special case with $g(x,\xi)=\nabla f(x,\xi)$. 
Although our main motivation is for solving the stochastic optimization problem~\eqref{eqn:stoch-opt} in the context of machine learning, the algorithms and convergence analysis developed in this paper apply to the general case of stochastic approximation with block-coordinate stepsizes. 

\subsection{Notations}
\label{sec:notations}

Let $\mathcal{I}_1,\ldots, \mathcal{I}_m$ be a non-overlapping partition of the index set $\{1,\ldots,n\}$, each with cardinality $n_k=|\mathcal{I}_k|$.
Correspondingly, we partition the vectors $x_t$, $s_t$ and $d_t$ into blocks 
$x_{t,k}$, $s_{t,k}$ and $d_{t,k}$ in $\R^{n_k}$ for $k=1,\ldots,m$.
%Here we use the simple notation $x_{t,k}$ instead of something like $x_{t,\mathcal{I}_k}$ or $x_{t,[k]}$, which would more clearly distinguish from the single coordinate notation of $x_{t,i}$.
We use a common stepsize $\gamma_{t,k}\in\R$ within each block, i.e.,
$s_{t,k} = \gamma_{t,k}\ones_{n_k}$, where $\ones_{n_k}$ denotes the vector of all ones of dimension~$n_k$.
As a result, the block-coordinate update in~\eqref{eqn:bc-sgd} can be written as
\[
x_{t+1,k} = x_{t,k} - s_{t,k}\odot d_{t,k} = x_{t,k} - \gamma_{t,k} d_{t,k}, \qquad k=1,\ldots,m.
\]
Notice that $\gamma_{t,k}$ is always a scalar and $\gamma_t$ is a vector in $\R^m$ instead of $\R^n$ (unless $m=n$).
%The following are the two extreme cases with block size $n$ and $1$ respectively:
%\begin{itemize}
%\item \emph{Full-dimension block:} $m=1$ and $\mathcal{I}_1 = \{1,\ldots,n\}$. We have a single stepsize $\gamma_t\in\R$ and $s_t=\gamma_t\odot\ones_n$. The algorithm becomes $x_{t+1} = x_t -\gamma_t d_t$, similar to~\eqref{eqn:sgd}.
%\item \emph{Single coordinate blocks}: $m=n$ and $\mathcal{I}_k = \{k\}$ for $k=1,\ldots,n$. In this case, we have $s_t=\gamma_t\in\R^n$ and can write $x_{t+1}=x_t-\gamma_t\odot d_t$, which is equivalent to~\eqref{eqn:bc-sgd}.
%\end{itemize}

Throughout this paper, $\langle\cdot,\cdot\rangle$ denotes the standard inner product in the Euclidean space~$\R^n$ and $\|\cdot\|$ denotes the induced Euclidean norm.
%For two matrices, $\langle A, B\rangle=\Tr(A^T B)$.
%Given a vector $\sigma\in\R^n$, $\diag(\sigma)$ denotes a diagonal matrix with $\sigma_i$ on its $i$th diagonal.
%
The $\sign(\cdot)$ function is defined as
$\sign(\alpha)=1$ if $\alpha>0$, $-1$ if $\alpha<0$ and $0$ if $\alpha=0$;
for vectors, it applies element-wise.

\section{Derivation of BCOS}
\label{sec:derivation}

In this section, we first derive the ideal optimal stepsizes for block-coordinate update, which unfortunately is not computable in practice; then we make several simplifications and approximations to derive a practical algorithm.

\subsection{Block-coordinate optimal stepsizes}
We consider the distance of the next iterate~$x_{t+1}$ to an (unknown) target point $x_*$. Specifically, in terms of the squared Euclidean distance, 
\begin{align*}
\|x_{t+1}-x_*\|^2
&= \|x_t -  s_t\odot d_t - x_*\|^2 \\
&= \|x_t-x_*\|^2 - 2\langle x_t-x_*, s_t\odot d_t\rangle + \| s_t\odot d_t\|^2 .
\end{align*}
Exploiting the block partitions of $x_t$, $s_t$ and $d_t$ (Section~\ref{sec:notations}) and $s_{t,k}=\gamma_{t,k}\ones_{n_k}$, we obtain 
\begin{align*}
\|x_{t+1}-x_*\|^2
&= \|x_t-x_*\|^2 + \sum_{k=1}^m \left(-2 \gamma_{t,k}\langle x_{t,k}-x_{*,k},\,d_{t,k}\rangle + \gamma_{t,k}^2\|d_{t,k}\|^2\right).
\end{align*}
Taking expectation conditioned on the realization of all random variables up to $x_t$, i.e.,  
\begin{equation}\label{eqn:cond-expect}
\E_t[\cdot]:=\E[\cdot|x_0,d_0,x_1,d_1, \ldots, x_t],
\end{equation}
we have
\begin{equation}\label{eqn:expected-dist}
\E_t\!\left[\|x_{t+1}-x_*\|^2\right]
= \|x_t-x_*\|^2 + \sum_{k=1}^m\Bigl(-2 \gamma_{t,k}\bigl\langle x_{t,k}-x_{*,k},\,\E_t[d_{t,k}]\bigr\rangle +  \gamma_{t,k}^2\E_t\!\left[\|d_{t,k}\|^2\right]\Bigr).
\end{equation}
In order to minimize the expected distance from $x_{t+1}$ to $x_*$, we can minimize the right-hand side of~\eqref{eqn:expected-dist} over the stepsizes $\{\gamma_{t,k}\}_{k=1}^m$. 
This can be done separately for each block by minimizing a simple quadratic, resulting in the optimal stepsizes
\begin{equation}\label{eqn:optimal-stepsizes}
\widehat\gamma_{t,k} = \frac{\langle x_{t,k}-x_{*,k},\, \E_t[d_{t,k}]\rangle}{\E_t[\|d_{t,k}\|^2]}, \qquad k=1,\ldots,m.
\end{equation}
Notice that these optimal stepsizes can be positive or negative, depending on the sign of the inner product in the numerator. 

Plugging the optimal stepsizes into~\eqref{eqn:expected-dist}, we have
\[
\E_t\!\left[\|x_{t+1}-x_*\|^2\right] = \|x_t-x_*\|^2
-\sum_{k=1}^m \frac{\left\langle x_{t,k}-x_{*,k},\,\E_t[d_{t,k}]\right\rangle^2}{\E_t[\|d_{t,k}\|^2]}.
\]
Therefore, we have a contraction of expected distance to~$x_*$ unless 
$$\langle x_{t,k}-x_{*,k},\,\E_t[d_{t,k}]\rangle=0,\qquad k=1,\ldots,m,$$
which means either $x_{t,k}=x_{*,k}$ or $\E_t[d_{t,k}]$ is orthogonal to $x_{t,k}-x_{*,k}$ simultaneously for all~$k$.
The latter case means that the random direction $d_t$ contains no information that can help reduce the distance. 

Apparently, the optimal stepsizes in~\eqref{eqn:optimal-stepsizes} are not computable in practice, because we do not have access to $x_*$ and cannot evaluate the expectations precisely. 
In the next section, we explain how to make approximations and derive a practical stepsize rule. 

\subsection{Simplification and approximation}
\label{sec:simplify-bcos}

We first try to avoid the direct reliance on~$x_*$, which is unknown to us.
To this end, rewrite the numerator in~\eqref{eqn:optimal-stepsizes} as
\[
\bigl\langle x_{t,k}-x_{*,k},\, \E_t[d_{t,k}]\bigr\rangle
 =\|x_{t,k}-x_{*,k}\| \|\E_t[d_{t,k}]\|\cos \theta_{t,k}, 
\]
where $\theta_{t,k}$ is the angle between the two vectors $x_{t,k}-x_{*,k}$ and $\E_t[d_{t,k}]$. 
We absorb the quantities related to $x_*$ into a tunable parameter
\begin{equation}\label{eqn:alpha-approx}
\alpha_{t,k}=\|x_{t,k}-x_{*,k}\|\cos\theta_{t,k},
\end{equation}
which leads to the stepsizes
\begin{equation}\label{eqn:bcos-1st-approx}
\widetilde\gamma_{t,k} = \frac{\alpha_{t,k}\|\E_t[d_{t,k}]\|}{\E_t[\|d_{t,k}\|^2]},
\qquad k=1,\ldots,m.
\end{equation}
Notice that any $\alpha_{t,k}$ we choose in practice may only be a (very rough) approximation of $\|x_{t,k}-x_{*,k}\|\cos\theta_{t,k}$.
In particular, while the optimal stepsizes $\widehat\gamma_{t,k}$ can be positive or negative, it is very hard to even correctly estimate the sign of the inner product
$\langle x_{t,k}-x_{*,k},\, \E_t[d_{t,k}]\rangle$.
We take the pragmatic approach of restricting $\alpha_{t,k}>0$, effectively being \emph{optimistic} that the expected direction $-\E_t[d_{t,k}]$ is always pointing towards $x_{*,k}$ (i.e., assuming $\cos\theta_{t,k}>0$).
%Therefore we call the resulting $\widetilde\gamma_{t,k}$ \emph{optimistic stepsizes}.

A further simplification is to use a common stepsize \emph{schedule} $\alpha_t$ across all blocks, i.e., $\alpha_{t,k}=\alpha_t$ for all $k=1,\ldots,m$.
This can be a reasonable choice for training deep neural networks, where the model parameters are initialized randomly coordinate-wise such that $\E[\|x_{0,k}\|]$ is constant for each coordinate $k$ \citep{HeZhangRenSun2015,RobertsYaidaHanin2022,DinanYaidaZhang2023}.
This brings us to 
\begin{equation}\label{eqn:bcos-conceptual}
\widetilde\gamma_{t,k} = \frac{\alpha_t\|\E_t[d_{t,k}]\|}{\E_t[\|d_{t,k}\|^2]},
\qquad k=1,\ldots,m.
\end{equation}
We note that with some abuse of notation, here $\alpha_t$ denotes a scalar, not a vector of $(\alpha_{t,1},\ldots,\alpha_{t,k})$.
This simplification motivates us to link $\alpha_t$ to the average of $\alpha_{t,k}$ in~\eqref{eqn:alpha-approx} and further relate to the overall distance $\|x_t-x_*\|$. 
Therefore, we expect~$\alpha_t$ to decrease as $\|x_t-x_*\|$ gradually shrinks. 
A simple strategy is to use a monotonic stepsize schedule on~$\alpha_t$, such as the popular cosine decay \citep{loshchilov2016sgdr} or linear decay \citep{defazio2024optimallineardecaylearning}.

Next, we need to replace the conditional expectations $\E_t[d_{t,k}]$ and $\E_t[\|d_{t,k}\|^2]$ in~\eqref{eqn:bcos-conceptual} with computable approximations. 
We adopt the conventional approach of \emph{exponential moving average} (EMA):
\begin{equation}\label{eqn:ema-estimators}
\begin{aligned}
u_{t,k} &= \beta u_{t-1,k} + (1-\beta) d_{t,k}, \\
v_{t,k} &= \beta v_{t-1,k} + (1-\beta) \|d_{t,k}\|^2,
\end{aligned}
\end{equation}
where $\beta\in[0,1)$ is the smoothing factor of EMA.
Notice that $u_{t,k}\in\R^{n_k}$ and $v_{t,k}\in\R$,
which have the same dimensions as $\E_t[d_{t,k}]$ and $\E_t[\|d_{t,k}\|^2]$ respectively.
This leads to a set of practical stepsizes:
\begin{equation}\label{eqn:bcos-uv}
\gamma_{t,k} = \alpha_t \frac{\|u_{t,k}\|}{v_{t,k}+\epsilon}, 
\qquad k=1,\ldots,m,
\end{equation}
where we added a small constant $\epsilon>0$ in the denominator to improve numerical stability (preventing overflow in case $v_{t,k}$ is too small).

\subsection{BCOS with one EMA estimator}
\label{sec:bcos-v}

The BCOS stepsizes in~\eqref{eqn:bcos-uv} are easy to implement in practice, but they are computed through the ratios of two online estimators $\|u_{t,k}\|$ and $v_{t,k}$.  The ratio of two random quantities may be susceptible to large fluctuations if the numerator and denominator drift in different directions (i.e., when one is much larger than its mean while the other is much smaller than its mean).
To avoid this problem, we further simplify BCOS to use only one EMA estimator.

First, recall the mean-variance decomposition of the conditional second moment:
\[
\E_t[\|d_{t,k}\|^2] = \|\E_t[d_{t,k}]\|^2 + \E_t[\|d_{t,k}-\E_t[d_{t,k}]\|^2] = \|\E_t[d_{t,k}]\|^2 + \Var_t(d_{t,k}),
\]
where $\Var_t(d_{t,k})$ denotes the conditional variance of $d_{t,k}$.
We interpret $\|\E_t[d_{t,k}]\|^2$ as the signal power and $\Var_t(d_{t,k})$ as the noise power, and define the \emph{signal fraction} (SiF) 
\begin{equation}\label{eqn:SiF-block}
\rho_{t,k} = \frac{\|\E_t[d_{t,k}]\|^2}{\E_t[\|d_{t,k}\|^2]}
= \frac{\|\E_t[d_{t,k}]\|^2}{\|\E_t[d_{t,k}]\|^2+\Var_t(d_{t,k})}.
\end{equation}
Apparently we have $\rho_{t,k}\in[0,1]$, and the two boundary values of~$0$ and~$1$ correspond to the cases of zero signal and full signal (zero noise) respectively.

With the definition of SiF, we can decompose the stepsizes in~\eqref{eqn:bcos-1st-approx} as
\[
\widetilde\gamma_{t,k} =
\alpha_{t,k}\frac{\|\E_t[d_{t,k}]\|}{\E_t[\|d_{t,k}\|^2]}
=\alpha_{t,k} \sqrt{\frac{\|\E_t[d_{t,k}]\|^2}{\E_t[\|d_{t,k}\|^2]}}
\frac{1}{\sqrt{\E_t[\|d_{t,k}\|^2]}}
=\frac{\alpha_{t,k} \sqrt{\rho_{t,k}}}{\sqrt{\E_t[\|d_{t,k}\|^2]}}.
\]
Now we can merge $\sqrt{\rho_{t,k}}\in[0,1]$ into the tunable parameters $\alpha_{t,k}$ and define 
$$\alpha'_{t,k}:=\alpha_{t,k}\sqrt{\rho_{t,k}}.$$
Then, following the same arguments as in Section~\ref{sec:simplify-bcos}, we choose to use a common stepsize schedule $\alpha'_t$ and obtain the practical stepsize rule
\begin{equation}\label{eqn:bcos-v}
\gamma_{t,k} = \alpha'_t \frac{1}{\sqrt{v_{t,k}+\epsilon}},
\qquad k=1,\ldots,m,
\end{equation}
where $v_{t,k}$ is the EMA estimator of $\E_t[\|d_{t,k}\|^2]$ given in~\eqref{eqn:ema-estimators}.

Another way to avoid using two estimators is to assimilate $\|\E_t[d_{t,k}]\|$, instead of $\sqrt{\rho_{t,k}}$, directly into the tunable parameter $\alpha'_{t,k}$ and ultimately $\alpha'_t$. This would result in
\begin{equation}\label{eqn:bcos-v2}
\gamma_{t,k} = \alpha'_t \,\frac{1}{v_{t,k}+\epsilon},
\qquad k=1,\ldots,m.
\end{equation}
There is a major advantage of~\eqref{eqn:bcos-v} over~\eqref{eqn:bcos-v2}.
Specifically, since the SiF $\rho_{t,k}$ is a dimensionless ratio, the per-iteration displacements,
\[
x_{t+1,k}-x_{t,k} = \alpha'_t \frac{d_{t,k}}{\sqrt{v_{t,k}+\epsilon}},
\]
are invariant of scaling of the search directions $d_{t,k}$, which makes the tuning of $\alpha'_t$ much easier in practice.
%We note that the two-estimator version of BCOS in~\eqref{eqn:bcos-uv} also has this invariance property, but the one in~\eqref{eqn:bcos-v2} does not.
This invariance property is also shared by Adam \citep{KingmaBa2014adam}, as seen from~\eqref{eqn:adam-stepsize} and~\eqref{eqn:adam-v}.
The similarity of Adam and the simplified BCOS in~\eqref{eqn:bcos-v} is apparent, and we will explain their connection in detail in the next section.

\iffalse
The above discussion also sheds light on the difference between BCOS/Adam and other coordinate-wise stepsize rules based on diagonal preconditioning.
Since the stochastic optimization problem~\eqref{eqn:stoch-opt} is essentially minimizing an infinite sum, the natural approach of preconditioning follows from the Gauss-Newton method rather than quasi-Newton type of methods.  
XXX Interpretation of the importance of square-root in the denominator. Compare with Gauss-Newton type of stepsizes and some recent proposals to remove the square-root. \citep{Bottou, Sophia, NGN}
In the convex case, AdaGrad, motivated by global convergence rates!
Nonconvex cases, mostly rely on the estimation of local curvature, mostly based on the Gauss-Newton method.
\fi

\section{Instantiations of BCOS}
\label{sec:instantiation}

The derivation of BCOS in Section~\ref{sec:derivation} is carried out with a general search direction $d_t$.
In this section, we instantiate BCOS with two common choices of the search direction: the stochastic gradient and its EMA, also known as stochastic momentum. 
% \todo{TJ: removed emph on ``stochastic momentum" since it was alr emphed on page 2}

To simplify presentation, we focus on the case of single-coordinate blocks, i.e., $m=n$ and $\mathcal{I}_k=\{k\}$ for $k=1,\ldots,n$.
%and $s_t=\gamma_t\in\R^n$.
Then we can write the coordinate-wise EMA estimators $v_{t,k}$ collectively in a vector form:
\begin{equation}\label{eqn:vector-ema}
v_t = \beta v_{t-1} + (1-\beta) d_t^2,
\end{equation}
where $d_t^2$ denotes the element-wise squared vector $d_t\odot d_t$.
We also have $s_t=\gamma_t\in\R^n$ and 
\[
x_{t+1}=x_t-\gamma_t\odot d_t
\]
where the vector of coordinate-wise stepsizes can be expressed as 
% \todo{Move $\epsilon$ under square root, to be consistent with proof!}
\begin{equation}\label{eqn:bcos-v-coord}
\gamma_t = \alpha_t \frac{1}{\sqrt{v_t+\epsilon}}
\end{equation}
Here $v_t+\epsilon$ means element-wise addition of~$\epsilon$, 
$\sqrt{v_t}+\epsilon$ denotes element-wise square root,
and the fraction represents element-wise division or reciprocal.
%Again, the stepsize schedule $\alpha_t$ is a scalar.
We no longer distinguish between $\alpha_t$ and $\alpha'_t$ because they are both tunable hyper-parameters.

\begin{figure}[t]
\begin{minipage}[t]{0.47\linewidth}
\begin{algorithm}[H]
\caption{BCOS-g (RMSprop)}
\label{alg:bcos-g}
\begin{algorithmic}[0]
\setstretch{1.4}
\State \textbf{input:} $x_0$, $\{\alpha_t\}_{t\geq 0}$, $\beta\in[0,1)$, $\epsilon>0$
%\State $v_{-1}=g_0^2$
\State $v_{-1}=\nabla f(x_0, \xi_0)^2$
\For{$t=0, 1,2, \ldots$}
  \State $g_t = \nabla f(x_t,\xi_t)$
  \State $v_t = \beta v_{t-1} + (1-\beta) g_t^2$
  %\State $x_{t+1} = x_t - \alpha_t \frac{1}{\sqrt{v_t + \epsilon}} \odot g_t$
  \State $\displaystyle x_{t+1} = x_t - \alpha_t \frac{g_t}{\sqrt{v_t + \epsilon}}$
\EndFor
\end{algorithmic}
\end{algorithm}
\end{minipage}\hfill
\begin{minipage}[t]{0.47\linewidth}
\begin{algorithm}[H]
\caption{BCOS-m}
\label{alg:bcos-m}
\begin{algorithmic}[0]
\setstretch{1.4}
\State \textbf{input:} $x_0$, $\{\alpha_t\}$, $\beta_1,\beta_2\in[0,1)$, $\epsilon>0$
%\State $m_{-1} = g_0$, ~$v_{-1}=g_0^2$
\State $m_{-1} = \nabla f(x_0,\xi_0)$
\State $v_{-1}=m_{-1}^2$
\For{$t=0, 1,2, \ldots$}
  \State $g_t = \nabla f(x_t,\xi_t)$
  \State $m_t = \beta_1 m_{t-1} + (1-\beta_1) g_t$
  \State $v_t = \beta_2 v_{t-1} + (1-\beta_2) \textcolor{blue}{m_t^2}$
  \State $\displaystyle x_{t+1} = x_t - \alpha_t \frac{m_t}{\sqrt{v_t + \epsilon}}$
\EndFor
\end{algorithmic}
\end{algorithm}
\end{minipage}
\end{figure}

\todo{Cite Robustness to Unbounded Smoothness of Generalized
SignSGD and its citations for BCOS-m}

\subsection{BCOS with EMA estimator}
\label{sec:bcos-ema}

We first present instantiations of BCOS with an EMA estimator (derived in Section~\ref{sec:bcos-v}).
%with an EMA estimator for the second moment of the search direction.

\paragraph{BCOS-g}
Algorithm~\ref{alg:bcos-g} is the instantiation of BCOS using $g_t=\nabla f(x_t,\xi_t)$ as the search direction, 
where $g_t^2$ denotes element-wise square 
%The vector $v_t$ consists of coordinate-wise EMA estimators for $\E[g_{t,k}^2]$
and $\frac{g_t}{\sqrt{v_t} + \epsilon}$ means element-wise division.
We call it BCOS-g to signify the use of gradient as search direction.

We immediately recognize that BCOS-g is exactly the RMSprop algorithm \citep{rmsprop2012lecture}, which is one of the first effective algorithms for training deep neural networks.
Our BCOS framework gives a novel interpretation of RMSprop and its effectiveness. 
In the special case with $\beta=0$ and $\epsilon=0$, we have $v_t=g_t^2$, and BCOS-g becomes the sign gradient method 
\begin{equation}\label{eqn:sign-gradient}
x_{t+1} = x_t - \alpha_t\, \sign(g_t),
\end{equation}
which also received significant attention in the literature 
\citep{Polyak1973pseudo,bernstein2018signsgd,safaryan2021signdescent,jiang2024efficient}.

\paragraph{BCOS-m}
Using the stochastic momentum as search direction has a long history in stochastic approximation \citep[e.g.,][]{GupalBazhenov72,Polyak77comparison,RuszczynskiSyski83stats}.
It has become the default option for modern deep learning due to its superior performance compared with using plain stochastic gradients.
In Algorithm~\ref{alg:bcos-m}, we use $m_t$ to denote the momentum, which is a standard notation in the machine learning literature.
We call it BCOS-m to signify the use of momentum as the search direction.
BCOS-m employs a second smoothing factor~$\beta_2$ to calculate the EMA of~$m_t^2$. The two smoothing factors $\beta_1$ and $\beta_2$ do not need to be the same and can be chosen independently in practice.
Similar to BCOS-g, the special case of BCOS-m with $\beta_2=0$ and $\epsilon=0$ corresponds to the sign-momentum method 
\begin{equation}\label{eqn:sign-momentum}
x_{t+1} = x_t - \alpha_t\, \sign(m_t).
\end{equation}

We notice that BCOS-m is very similar to Adam as given in~\eqref{eqn:adam-stepsize} and~\eqref{eqn:adam-v}.
The difference is that in Adam, $v_t$ is the EMA of~$g_t^2$ instead of~$m_t^2$.
From the perspective of the BCOS framework, Adam has a mismatch between the search direction $m_t$ and the second moment estimator based on $g_t^2$, which must be compensated for by using a larger smoothing factor~$\beta_2$ (because $m_t$ itself is a smoothed version of $g_t$).  
For BCOS-m, using $\beta_2=\beta_1$ produces as good performance as Adam with the best tuned $\beta_2$ (see numerical experiments in Section~\ref{sec:experiments}).

\bigskip

In both BCOS-g and BCOS-m, we initialize $v_{-1}=g_0^2$, therefore $v_0=g_0^2$ as well. This avoids the initial bias caused by setting $v_{-1}=0$. 
An alternative way for bias correction is to initialize $m_{-1}$ and $v_{-1}$ as the zero vector, and then multiply $m_t$ and $v_t$ by $1/(1-\beta^{t+1})$, which quickly converges to~$1$, as done in Adam~\citep{KingmaBa2014adam}. Our numerical experiments show that these two choices do not have significant differences in practice.

\subsection{BCOS with conditional estimator}
\label{sec:bcos-conditional}

Recall that the derivations of optimal stepsizes $\widehat{\gamma}_{t,k}$ and $\widetilde{\gamma}_{t,k}$ in Section~\ref{sec:derivation} are all based on \emph{conditional} expectation $\E_t[\cdot]$.
In Section~\ref{sec:bcos-ema}, we used coordinate-wise EMA to approximate the conditional expectations, specifically, $v_t$ as an estimator of $\E_t[d_t^2]$ in BCOS-g and an estimator of $\E_t[m_t^2]$ in BCOS-m.  In this section, we show that with $m_t$ as the search direction, we can exploit its update form to derive more effective \emph{conditional estimators}.

We first repeat the definition of momentum with a smoothing factor $\beta\in[0,1)$, i.e., 
\begin{equation}\label{eqn:momentum}
m_t = \beta m_{t-1} + (1-\beta) g_t.
\end{equation}
To derive an estimator of $\E_t[m_t^2]$, we expand the square and take the conditional expectation of each term:
\begin{align}
\E_t\bigl[m_t^2\bigr] 
&= \E_t\bigl[(\beta m_{t-1}+(1-\beta) g_t)^2\bigr] \nonumber\\
&= \beta^2 \E_t\bigl[m_{t-1}^2\bigr] + 2\beta(1-\beta)\E_t[ m_{t-1}\odot g_t]
+ (1-\beta)^2 \E_t\bigl[g_t^2\bigr]  \nonumber\\
&= \beta^2 m_{t-1}^2 + 2\beta(1-\beta)m_{t-1}\odot \E_t[g_t]
+ (1-\beta)^2 \E_t\bigl[g_t^2\bigr],
\label{eqn:Et-m2}
\end{align}
where we used the fact $\E_t[m_{t-1}^2]=m_{t-1}^2$ and $\E_t[m_{t-1}]=m_{t-1}$ thanks to the definition of $\E_t[\cdot]$ in~\eqref{eqn:cond-expect}.
It remains to approximate $\E_t[g_t]$ and $\E_t[g_t^2]$.

Clearly a good estimator for $\E_t[g_t]$ is~$m_t$.
To approximate $\E_t[g_t^2]$, we could use a separate EMA estimator 
$v'_t=\beta'v'_{t-1}+(1-\beta')g_t^2$, 
but this introduces another algorithm state $v'_t$ and a second smoothing factor~$\beta'$.
Meanwhile, we observe that the factor $(1-\beta)^2$ multiplying $\E_t[g_t^2]$ is usually very small, especially for $\beta$ close to~1.
As a result, any error in approximating $\E_t[g_t^2]$ is attenuated by a very small factor and may not cause much difference.
Therefore, for simplicity, we choose to approximate $\E_t[g_t^2]$ with $g_t^2$ itself. 
Combining with approximating $\E_t[g_t]$ with $m_t$, we arrive at the following \emph{conditional} estimator for $\E_t[m_t^2]$:
\begin{equation}\label{eqn:v-conditional}
v_t = \beta^2 m_{t-1}^2 + 2\beta(1-\beta) m_{t-1} \odot m_t + (1-\beta)^2 g_t^2.
\end{equation}
The estimator in~\eqref{eqn:v-conditional} does not need to store $v_{t-1}$, thus requiring less memory (fewer optimizer states). Moreover, this estimator eliminates $\beta_2$ as a second hyper-parameter. 

\begin{algorithm}[t]
        \caption{BCOS-c}
        \label{alg:bcos-c}
        \begin{algorithmic}[0]
        \setstretch{1.4}
        \State \textbf{input:} $x_0$, $\{\alpha_t\}$, $\beta\in[0,1)$, $\epsilon> 0$, 
        \State $m_{-1} = \nabla f(x_0, \xi_0)$ 
        \For{$t=0, 1,2, \ldots$}
          \State $g_t = \nabla f(x_t,\xi_t) $ 
          \State $m_t = \beta m_{t-1} + (1-\beta) g_t$
          %\State Conditional estimator: 
          \State $v_t = \beta^2 m^2_{t-1} + 2 \beta (1-\beta) m_{t-1} \odot m_t + (1-\beta)^2 g^2_t$ 
          \State $\Big($Simple alternative: $v_t = \bigl(1-(1-\beta)^2\bigr) m^2_{t-1} + (1-\beta)^2 g^2_t$ $\Big)$
          % \State $v_t = \bigl(1-(1-\beta)^2\bigr) m^2_{t-1} + (1-\beta)^2 g^2_t$
        %  \State $v_t = \beta^2 m_{t-1}^2 + 2\beta(1-\beta)m_{t-1} \odot m_t + (1-\beta)^2 g_t^2$
          \State $\displaystyle x_{t+1} = x_t - \alpha_t \frac{m_t}{\sqrt{v_t + \epsilon}}$
        \EndFor
        \end{algorithmic}
\end{algorithm}

Replacing $v_t$ in BCOS-m with~\eqref{eqn:v-conditional} leads to Algorithm~\ref{alg:bcos-c}. We call it BCOS-c to signify the \emph{conditional} estimator. 

\iffalse
\begin{algorithm}[t]
        \caption{BCOS\textcolor{blue}{W}-c}
        \label{alg:bcosw-c}
        \begin{algorithmic}[0]
        \setstretch{1.4}
        \State \textbf{input:} $x_0$, $\{\alpha_t\}$, $\beta\in[0,1)$, $\epsilon> 0$, \textcolor{blue}{$\lambda\geq 0$}
        %\State $m_{-1} = g_0$ 
        \State $m_{-1} = \nabla f(x_0, \xi_0)$ 
        \For{$t=0, 1,2, \ldots$}
          \State $g_t = \nabla f(x_t,\xi_t) $ 
          \State $m_t = \beta m_{t-1} + (1-\beta) g_t$
          \State $v_t = \beta^2 m^2_{t-1} + 2 \beta (1-\beta) m_{t-1} \odot m_t + (1-\beta)^2 g^2_t$ 
          %\State $v_t = (3-2\beta)\beta^2 m_{t-1}^2 + 2\beta(1-\beta)^2 m_{t-1}\odot g_t + (1-\beta)^2 g_t^2$
          \State $\Big($Simple alternative: $v_t = \bigl(1-(1-\beta)^2\bigr) m^2_{t-1} + (1-\beta)^2 g^2_t$ $\Big)$
          \State $\displaystyle x_{t+1} = \textcolor{blue}{(1-\alpha_t\lambda)}x_t - \alpha_t \frac{m_t}{\sqrt{v_t + \epsilon}}$
        \EndFor
        \end{algorithmic}
\end{algorithm}
\fi

% NO NEED TO EXPLAIN HERE
%We note the slight abuse of notation in using the symbol $v_t$ for both the EMA estimator in BCOS-g and -m and the conditional estimator here. In general, $v_t$ denotes an online estimator of $\E_t[d_t^2]$. The corresponding choice of~$d_t$ and type of estimator should be clear from the context.

The estimator in~\eqref{eqn:v-conditional} can be further simplified if we approximate $\E_t[g_t]$ in~\eqref{eqn:Et-m2} with $m_{t-1}$ instead of~$m_t$, which results in
\begin{align}
v_t &= \beta^2 m_{t-1}^2 + 2\beta(1-\beta) m_{t-1}^2  + (1-\beta)^2 g_t^2 \nonumber \\
& = \left(1-(1-\beta)^2\right) m^2_{t-1} + (1-\beta)^2 g^2_t .
\label{eqn:v-cond-simple}
\end{align}
This estimator \emph{resembles} the standard EMA estimator in Adam, shown in~\eqref{eqn:adam-v}, with an effective smoothing factor
\begin{equation}\label{eqn:v-cond-simple-beta}
\beta'=1-(1-\beta)^2,
\end{equation}
\emph{but with $v_{t-1}$ replaced by $m^2_{t-1}$.} This also explains that the second smoothing factor in Adam, $\beta_2$, corresponding to~$\beta'$ here, should be much larger or closer to~1 than~$\beta$.
Specifically, $\beta=0.9$ roughly corresponds to $\beta'=0.99$. 

We include the estimator~\eqref{eqn:v-cond-simple} as a simple alternative of~\eqref{eqn:v-conditional} in Algorithm~\ref{alg:bcos-c}. 
Empirically we found that it behaves almost identically as~\eqref{eqn:v-conditional} on the experiments we conduct in Section~\ref{sec:experiments}. 
Concurrently, \citet{zhang2025adams} propose a similar algorithm with 
\[
v_t = \beta' m_{t-1}^2 + (1-\beta') g_t^2, 
\]
and show that tuning $\beta'$ as a hyper-parameter can obtain slightly better performance than using the value in~\eqref{eqn:v-cond-simple-beta}.
% \todo{We should compare tuning $\beta'$ with BCOS-c on larger models. But don't have time now.}

%%%%%%%%%%%%%%%%%%%%%%%%%%%%%%%%%%%%%%

\begin{algorithm}[t]
        \caption{BCOSW\{-g, -m, -c\}}
        \label{alg:bcosw-c}
        \begin{algorithmic}[0]
        \setstretch{1.4}
        \State \textbf{input:} $x_0$, $\{\alpha_t\}$, $\beta\in[0,1)$, $\epsilon> 0$, \textcolor{blue}{$\lambda\geq 0$}
        \For{$t=0, 1,2, \ldots$}
          \State compute $d_t$ ($g_t$ or $m_t$) and $v_t$ as in Algorithm~\ref{alg:bcos-g}, \ref{alg:bcos-m} or \ref{alg:bcos-c} (-g, -m or -c, respectively)
          \State $\displaystyle x_{t+1} = \textcolor{blue}{(1-\alpha_t\lambda)}x_t - \alpha_t \frac{d_t}{\sqrt{v_t + \epsilon}}$
        \EndFor
        \end{algorithmic}
\end{algorithm}

\subsection{BCOS with decoupled weight decay}
\label{sec:bcosw}

Weight decay is a common practice in training deep learning models to obtain better generalization performance.
It can be understood as adding an $L_2$ regularization to the loss function, i.e., minimizing the regularized loss 
\begin{equation}\label{eqn:l2-regu}
F(x)=\E_\xi[f(x,\xi)]+\frac{\lambda}{2}\|x\|^2,
\end{equation}
where $\lambda>0$ is a regularization parameter.
Effectively, the stochastic gradient at $x_t$ becomes
$\nabla f(x_t,\xi_t)+\lambda x_t$.
We can apply the BCOS family of algorithms to incorporate the regularization term by simply replacing $g_t=\nabla f(x_t,\xi_t)$ with $g_t=\nabla f(x_t,\xi_t)+\lambda x_t$.

A more effective way in practice is to use \emph{decoupled weight decay} as proposed in the AdamW algorithm \citep{Loshchilov2019AdamW}. 
Specifically, we apply weight decay separately in the BCOS update: 
\[
x_{t+1} = x_t - \gamma_t\odot d_t - \alpha_t\lambda x_t =(1-\alpha_t\lambda)x_t - \gamma_t\odot d_t.
\]
We call the resulting method BCOSW following the naming convention of AdamW.

Algorithm~\ref{alg:bcosw-c} shows the BCOSW framework, which can be instantiated with specific $d_t$ and $v_t$ as in BCOS-g, -m or -c.
A PyTorch implementation of all variants is available at 
% \todo{need to update the repo to include the non-simplified BCOS-c}
\centerline{\url{https://github.com/facebookresearch/bcos}.}

%For implementation (without needing both $m_{t-1}$ and $m_t$ in memory):
%\[
%v_t = (3-2\beta)m_t^2 - 2(1-\beta)(2-\beta)m_t\odot g_t + 2(1-\beta)^2 g_t^2
%\]
%or
%\[
%v_t = (3-2\beta)\beta^2 m_{t-1}^2 + 2\beta(1-\beta)^2 m_{t-1}\odot g_t + (1-\beta)^2g_t^2
%\]

%%%%%%%%%%%%%%%%%%%%%%%%%%%%%%%%%%%%%%

\section{Numerical experiments}
\label{sec:experiments}

%cosine decay
%\[
%\alpha_t = \alpha_0 \frac{\cos (\pi t/T)+1}{2}
%\]

In this section, we present preliminary experiments to compare BCOS with Adam, specifically their variants with decoupled weight decay: AdamW and BCOSW.
Among the BCOSW family, we focus on BCOSW-c. 

Our first set of experiments is conducted on training a small GPT2 model with 124 million parameters \citep{radford2019language} on the OpenWebText dataset \citep{Gokaslan2019OpenWeb}.
We use global batch size 512 and run all experiments for 100k iterations.
We use linear warmup on the stepsize schedule $\{\alpha_t\}$ to reach $\alpha_\text{max}$ in the first 2k iterations, and then use cosine decay to toward the final stepsize $\alpha_\text{min}=0.01\alpha_\text{max}$.
The default hyper-parameters are chosen (based on a coarse sweep) as: peak stepsize $\alpha_\text{max}=0.002$, $\epsilon=10^{-12}$ and weight decay $\lambda=0.1$. 

\begin{figure}[t]
    \centering
\ifloadfigs
    \input{figs/adamw_bcosw_beta.tex}
\else
    \includegraphics[width=0.99\textwidth]{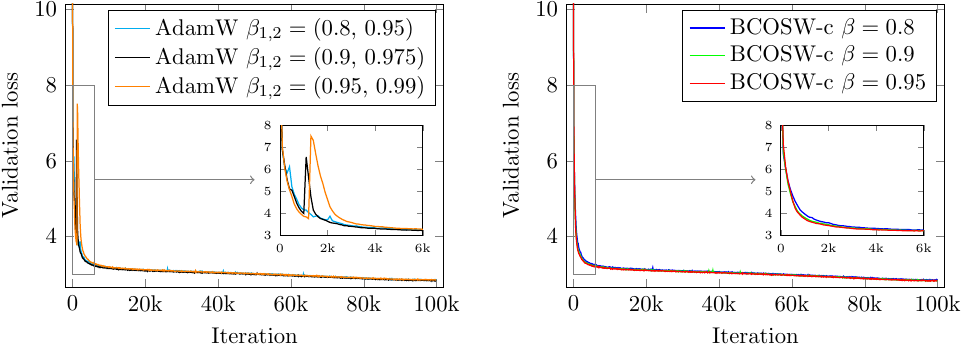}
\fi
\caption{Comparing AdamW and BCOSW-c with different momentum parameters.}
\label{fig:adamw-bcosw-beta}
\end{figure}

Figure~\ref{fig:adamw-bcosw-beta} (left) shows the test loss of AdamW on the GPT2 task with different combinations of~$\beta_1$ and~$\beta_2$. 
For each value of $\beta_1\in\{0.8, 0.9, 0.95\}$, we choose the best $\beta_2$ after sweeping $\beta_2\in\{0.8, 0.9, 0.95, 0.975, 0.99\}$. 
Their final losses achieved are all very close, around the value $2.82$. 
For most $(\beta_1,\beta_2)$ combinations, we observe some loss spikes, especially at the beginning of the training (as shown in the inset).
In contrast, Figure~\ref{fig:adamw-bcosw-beta} (right) shows that BCOSW-c obtains the same final loss but with very smooth loss curves.
%The corresponding training losses are very similar and shown in Figure~\ref{fig:adamw-bcosw-beta-train} in Appendix~\ref{sec:more-experiments}.

\begin{figure}[t]
    \centering
\ifloadfigs
    \input{figs/adamw_bcosw_gmc.tex}
\else
    \includegraphics[width=0.99\textwidth]{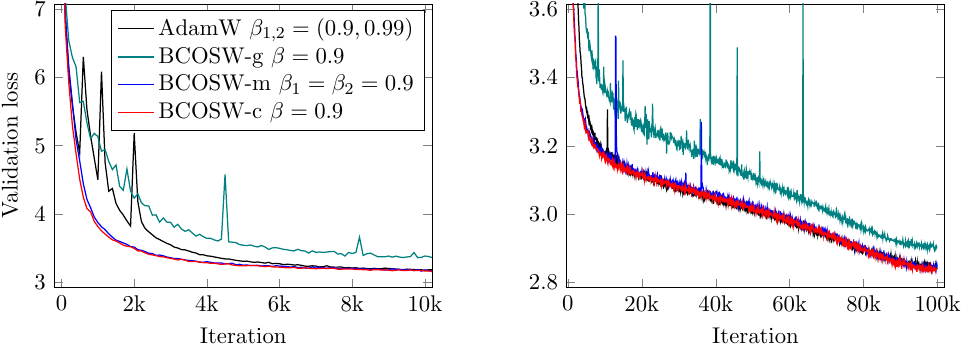}
\fi
\caption{Comparing AdamW and three variants of BCOSW. 
Left: the first 10k iterations; Right: all 100k iterations.
}
\label{fig:adamw-bcosw-gmc}
\end{figure}

%%%%%%%%%%%%%%%%%%%%%%%%%%%%%%%%%%%%%%
\begin{figure}[p]
    \centering
\ifloadfigs
    \input{figs/adam_bcos_wd.tex}
\else
    \includegraphics[width=0.97\textwidth]{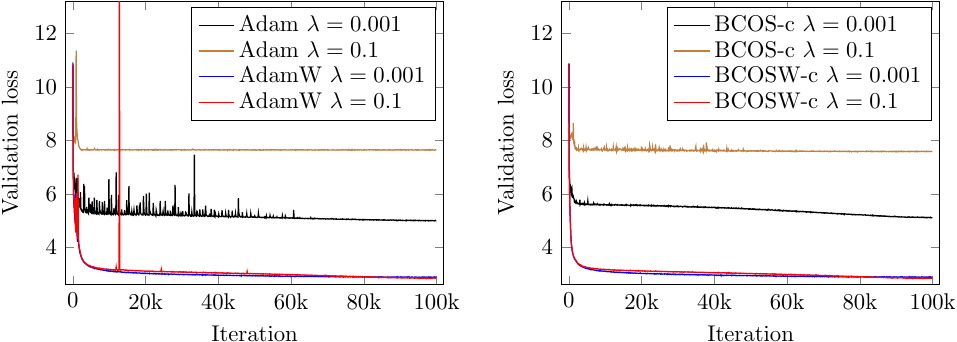}
\fi
\caption{Left: Adam(W) with $\beta_1=0.9$ and $\beta_2=0.99$. Right: BCOS(W)-c with $\beta=0.9$.}
\label{fig:adam-bcos-wd}
\vspace{2ex}
\end{figure}

%%%%%%%%%%%%%%%%%%%%%%%%%%%%%%%%%%%%%%
\begin{figure}[p]
    \centering
\ifloadfigs
  \input{figs/adam_bcosw_lr.tex}
\else
    \includegraphics[width=0.97\textwidth]{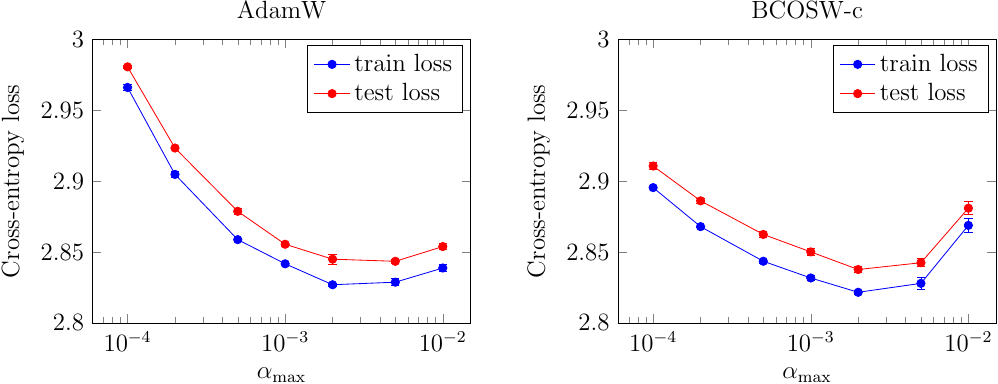}
\fi
\caption{Train and test loss by varying max value of stepsize schedule.}
\label{fig:adamw-bcosw-lr}
\vspace{2ex}
\end{figure}

%%%%%%%%%%%%%%%%%%%%%%%%%%%%%%%%%%%%%%
\begin{figure}[p]
    \centering
\ifloadfigs
    \input{figs/resnet_vit.tex}
\else
    \includegraphics[width=0.97\textwidth]{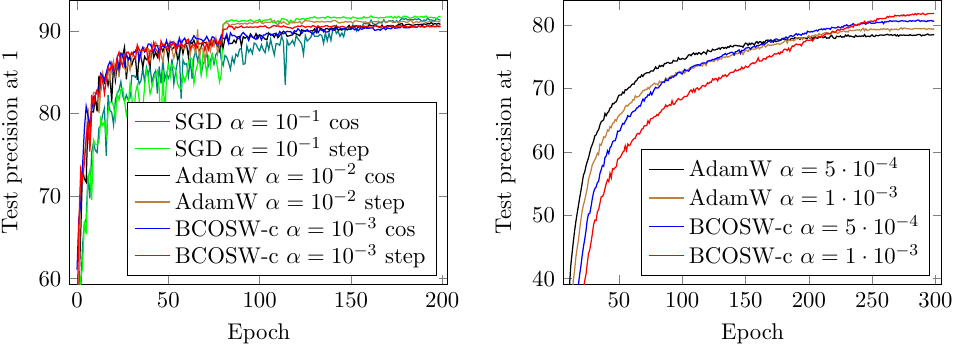}
\fi
\caption{Left: ResNet-20 on CIFAR10. Right: Vision Transformer on ImageNet.}
\label{fig:resnet-vit}
\end{figure}

Figure~\ref{fig:adamw-bcosw-gmc} compares AdamW against the three variants BCOSW-g, -m, and -c.
We observe that BCOSW-g is significantly worse than the momentum-based methods.
The loss curves for the momentum-based methods are all very close, but with spikes for both AdamW and BCOSW-m, while BCOSW-c has very smooth loss curves.

Figure~\ref{fig:adam-bcos-wd} illustrates the difference between algorithms with and without decoupled weight decay. 
For Adam and BCOS-c, the weight decay is accounted for in the gradient with $g_t=\nabla f(x_t, \xi_t)+\lambda x_t$. 
BCOS-c converges to much higher loss than BCOSW-c, and different values of $\lambda$ (weight decay) make a dramatic difference for BCOS-c but cause little change to BCOSW-c.
The same phenomenon happens for Adam versus AdamW, and we again observe more spikes from the loss curves of Adam(W) than BCOS(W)-c.
%(with different random seed than in Figures~\ref{fig:adamw-bcosw-beta} and~\ref{fig:adamw-bcosw-gmc}).

Figure~\ref{fig:adamw-bcosw-lr} shows how the final training and validation losses (after 100k iterations) vary with $\alpha_\text{max}$.
For each stepsize schedule, we repeat the training three times with different random seeds and plot the standard deviation as error bars, which are mostly invisible due to small variations among different runs.
We observe that AdamW and BCOSW obtain the lowest losses at similar values of $\alpha_\text{max}$. 
BCOSW has slightly better performance for smaller values of $\alpha_\text{max}$, but becomes more sensitive for larger values of $\alpha_\text{max}$.

Finally, in Figure~\ref{fig:resnet-vit}, we compare different algorithms for training ResNet-20 \citep{he2016resnet} on the CIFAR10 dataset \citep{krizhevsky2009learning}, and also training the Vision Transformer (ViT) \citep{touvron21vit} on the ImageNet dataset \citep{imagenet2009}. 
For the ResNet task, we tried both cosine decay (drop by factor 100) and step decay (drop by 10 at epochs 80, 120, 150).  
The hyper-parameters chosen are: $\beta=0.9$ for SGD and BCOSW-c, and $\beta_{1,2}=(0.9,\,0.99)$ for AdamW. 
We observe that the best-performing stepsize schedules are quite different for different methods. This prompts the need to tune hyper-parameters for BCOSW for different tasks even though it shares similar tuned hyper-parameters as AdamW on the GPT2 task. 

For training ViT on ImageNet, although the best tuned stepsize schedules are similar between AdamW and BCOSW, their training and test curves look quite different. 
Figure~\ref{fig:resnet-vit} (right) shows that the test precision curves for BCOSW-c rise slowly but reach slightly higher precision at the end.  

These preliminary experiments demonstrate that BCOSW-c can obtain competitive performance compared with the state-of-the-art method AdamW, but require less memory and fewer hyper-parameters to tune.
More comprehensive empirical studies are needed to fully understand its potential in training larger deep learning models.

%%%%%%%%%%%%%%%%%%%%%%%%%%%%

%Concurrently, \citep{zhang2025adams} proposed an identical Algorithm~\ref{alg:bcos-c} with the estimator in~\eqref{eqn:v-cond-simple}. In their version, the smoothing coefficient in $v_t$ can be decoupled from $\beta$. The experiments in~\citep{zhang2025adams} ranged from pre-training tasks including Llama2-7B, GPT2-small, -medium, -large, to RL post-traing tasks such as QWen2.5-3B and DeepSeek-R1-Distill-Llama-8B,  demonstrating the effectiveness of conditional estimators~\eqref{eqn:v-conditional} independently.
%\todo{Removed the discussion on simple-c and made the ugly version the prominent algorithm. Added concurrent work. In small models, (26) performs similarly as (25). But for larger model, the -c is better. New comments: the AdamS paper demonstrates that BCOS-c can be as good as Adam on Llama2-7B, so still working on larger models.} 

\section{Convergence analysis}
\label{sec:convergence}

In this section, we present convergence analysis for BCOS(W).
To simplify presentation, we stick with the setting of single-coordinate blocks. Specifically, we analyze the algorithm
\begin{equation}\label{eqn:bcosw-practical}
x_{t+1}=(1-\alpha_t\lambda) x_t -\gamma_t \odot d_t, 
\quad\text{where}\quad
\gamma_t=\frac{\alpha_t}{\sqrt{v_t+\epsilon}},
\end{equation}
%Here, stepsize $\gamma_t$ takes a slightly different form from the one used throughout Section~\ref{sec:instantiation}, where $\epsilon$ is outside the square-root in the denominator. The resulting difference between the two forms is negligible for corresponding choices of $\epsilon$, with the one in~\eqref{eqn:bcosw-practical} being squared (e.g., being $10^{-12}$ if the one in Section~\ref{sec:instantiation} is $10^{-6}$). We choose the form in~\eqref{eqn:bcosw-practical} for ease of presenting the analysis. 
%
where $d_t$ is a general search direction and $v_t\in\R_+^n$ is an online estimator of $\E_t[d_t^2]$.
Our analysis consists of two stages:
\begin{itemize}
\item
First, we analyze the convergence behavior of the \emph{conceptual} BCOS method
\begin{equation}\label{eqn:bcosw-conceptual}
x_{t+1} = (1-\alpha_t\lambda)x_t - \widetilde\gamma_t\odot d_t,
\quad\text{where}\quad
\widetilde\gamma_t=\frac{\alpha_t}{\sqrt{\E_t[d_t^2]}}.
\qquad
\end{equation}
It is called ``conceptual'' because we cannot compute $\E_t[d_t^2]$ exactly in practice. But its analysis is essential for understanding and analyzing the practical method.
Our analysis relies on an \emph{aiming condition}, which originates from the classical stochastic approximation literature 
\citep{RobbinsMonro51,Blum1954multidim,chung1954,Sakrison1966}.
Under this assumption, we establish \emph{almost sure} convergence of $x_t$ to a target point~$x_*$ and the $\cO(1/t)$ convergence rate of $\E[\|x_t-x_*\|^2]$.
\item
To analyze the practical BCOS method~\eqref{eqn:bcosw-practical}, we first bound the difference between the expected steps under the practical method and the conceptual method, i.e., 
\begin{equation}\label{eqn:step-diff}
\bigl|\E_t[\gamma_t\odot d_t] - \E_t[\widetilde\gamma_t\odot d_t]\bigr|
\leq \sigma_t \bigl|\E_t[\widetilde\gamma\odot d_t]\bigr| + \cO(\epsilon),
%+ \cO(\Var_t(v_t)),
\end{equation}
where $\leq$ denotes coordinate-wise inequalities between two vectors and $\sigma_t$ is a scalar that depends on the bias and variance of~$v_t$ as an online estimator of $\E_t[d_t^2]$.
We prove almost sure convergence of $x_t$ to a neighborhood of~$x_*$, where the radius of the neighborhood is determined by the magnitude of~$\sigma_t$ and~$\cO(\epsilon)$.
%the right-hand side of~\eqref{eqn:step-diff}.
% \todo{Need to change notation: $c_t$ to $\theta_t$, $\sigma_t$?}
\end{itemize}

The rest of this section is organized as follows.
In Section~\ref{sec:analysis-aiming}, we introduce the aiming condition and give some empirical evidence. 
% \todo{TJ: since we have moved the discussion on convexity to appendix, shall we change this to ``check the aiming condition empirially on GPT2"}
We present the convergence analysis for the conceptual BCOS method in Section~\ref{sec:analysis-conceptual} and then the analysis for the practical BCOS method in Section~\ref{sec:analysis-practical}. 
Our analysis is very general and applies to the whole BCOS family of methods, including the sign-gradient method, RMSprop and Adam(W). 
To obtain specific results for each method, we only need to estimate the bias and variance of the second-moment estimator $v_t$, which are presented in Section~\ref{sec:bias-var-tradeoff}.

\subsection{Aiming condition}
\label{sec:analysis-aiming}

First consider the stochastic approximation method~\eqref{eqn:general-sa}.
A typical set of assumptions in the classical stochastic approximation literature includes:
%\citep[e.g.,][]{Sakrison1966} includes: 
\begin{itemize} 
\item The random variables $\{\xi_t\}_{t\geq 0}$ are independent and $g(x,\xi)$ has bounded variance;
\item The step sizes $\{\alpha_t\}_{t\geq 0}$ are nonnegative and satisfy
$\sum_{t=0}^\infty \alpha_t = \infty$ and $\sum_{t=0}^\infty \alpha_t^2<\infty$;
\item Suppose $\E_\xi[g(x_*,\xi)]=0$ and there exists $L,\mu>0$ such that for all $x\in\R^n$,
\begin{equation}\label{eqn:aiming-sakrison}
L\|x-x_*\|^2 ~\geq~ \bigl\langle x - x_*,\, \E_\xi[g(x,\xi)] \bigr\rangle
~\geq~ \mu\|x-x_*\|^2. 
\end{equation}
\end{itemize}
Under these conditions, the sequence $\{x_t\}$ converges to $x_*$ in the mean-square sense, i.e., $\lim_{t\to\infty}\E[\|x_t-x_*\|^2]=0$ \citep[e.g.,][]{Sakrison1966}.
In particular, the last condition~\eqref{eqn:aiming-sakrison} ensures that the direction $-\E_\xi[g(x,\xi)]$ always points towards the target $x_*$, implying that the distance from~$x_*$ can be reduced along this direction (with appropriate stepsize).
For this reason, we call it an \emph{aiming} condition.
% \todo{I moved the nonconvex example to Appendix because it does not satisfy upper bound here.}

In the context of stochastic optimization, i.e., minimizing $F(x)=\E_\xi[f(x,\xi)]$ with $g(x,\xi)=\nabla f(x,\xi)$ and $\E_\xi[g(x,\xi)]=\nabla F(x)$, the aiming condition~\eqref{eqn:aiming-sakrison} is implied by~$F$ being $\mu$-strongly convex and $L$-smooth \citep[e.g.,][]{Nesterov04book}.
But in general it is a much weaker condition.
Moreover, the upper bound on the inner product can often be removed by leveraging more sophisticated analysis \citep[e.g.,][]{RobbinsMonro51, Wolfowitz1952} or careful choice of the search direction $g(x,\xi)$ \citep[e.g.][]{Blum1954multidim, Dvoretzky1956}. This is the case of our aiming condition, which we present next.

In order to allow reduction of the expected distance from $x_{t+1}$ to $x_*$, we rewrite the conceptual BCOS method~\eqref{eqn:bcosw-conceptual} as
\[
x_{t+1} = x_t - \alpha_t\lambda x_t - \widetilde\gamma_t\odot d_t
= x_t - \alpha_t\left(\frac{d_t}{\sqrt{\E_t[d_t^2]}}+\lambda x_t\right)
\]
and impose an aiming condition on the expected update direction
\[
\E_t\!\left[\frac{d_t}{\sqrt{\E_t[d_t^2]}}+\lambda x_t\right] = \frac{\E_t[d_t]}{\sqrt{\E_t[d_t^2]}}+\lambda x_t.
\]
\begin{assumption}[Aiming condition]\label{assum:aiming}
There exists $x_*\in\R^n$ such that 
\begin{equation}\label{eqn:aiming}
\left\langle x_t-x_*, \, \frac{\E_t[d_t]}{\sqrt{\E_t[d_t^2]}} + \lambda x_t\right\rangle \geq \lambda \|x_t - x_*\|^2
\end{equation}
holds for all $t\geq 0$ almost surely.
%If $d_t$ is independent of the past trajectory conditioned on~$x_t$, i.e., $\E_t[d_t]=\E[d_t|x_t]$, then it suffices to have~\eqref{eqn:aiming} hold for every $x\in\R^n$ (independent of the trajectory).
\end{assumption}

We make the following remarks regarding Assumption~\ref{assum:aiming}.
\begin{itemize}
\item \label{rmk:aiming-trajectory}
The condition~\eqref{eqn:aiming} in general depends on the past trajectory $\{x_0,d_0,x_1,d_1,\ldots,x_t\}$. 
%through the definition of $\E_t[\cdot]$ in~\eqref{eqn:cond-expect}.
If $\E_t[d_t]=\E[d_t|x_t]$, then we can drop the subscript~$t$ and obtain a trajectory-independent condition like~\eqref{eqn:aiming-sakrison}.
This is indeed the case when $d_t=g(x_t,\xi_t)$ and $\{\xi_0,\ldots,\xi_t\}$ are mutually independent.
We need~\eqref{eqn:aiming} to cover more general cases, e.g., when $d_t$ is the stochastic momentum, which clearly depends on the past trajectory.
\item \label{rmk:aiming-lambda} 
Instead of a generic constant~$\mu$ as in~\eqref{eqn:aiming-sakrison}, we have $\lambda$ on the right-hand side of~\eqref{eqn:aiming} due to the use of decoupled weight decay, interpreted as $L_2$-regularization in~\eqref{eqn:l2-regu}.
\item \label{rmk:aiming-no-upper} 
We do not need an upper bound on the inner product as in~\eqref{eqn:aiming-sakrison} due to the normalization effect of dividing~$d_t$ by $\sqrt{\E_t[d_t^2]}$ in BCOS.
A similar one-sided aiming condition was used in \citep{Blum1954multidim} to establish almost sure convergence of the first multi-dimensional extension of the Robbins-Monro method. 
\end{itemize}

Let $\rho_t\in[0,1]^n$ be the vector of coordinate-wise SiF defined in~\eqref{eqn:SiF-block}. Then we have  
\[
\frac{\E_t\left[d_t\right]}{\sqrt{\E_t[d_t^2]}} = \sqrt{\frac{\E_t[d_t]^2}{\E_t[d_t^2]}}\,\sign(\E_t[d_t]) = \sqrt{\rho_t}\odot \sign(\E_t[d_t]).
\]
Correspondingly, the aiming condition~\eqref{eqn:aiming} can be written as
\begin{equation}\label{eqn:aiming-sif}
\bigl\langle x_t-x_*, \, \sqrt{\rho_t}\odot\sign(\E_t[d_t])+\lambda x_t\bigr\rangle \geq \lambda\|x_t-x_*\|^2.
\end{equation}
Intuitively, the inner product in the aiming condition is predominantly determined by the coordinates that are far from target, in other words, the coordinates with relatively large $|x_{t,k}-x_{*,k}|$ and having a large SiF $\rho_{t,k}$ (effectively high signal-to-noise ratio).

In general we cannot check the aiming condition analytically, since it involves conditional expectations and an unknown target point~$x_*$.
However, we can check empirically if the inner product 
$\langle x_t-x_*, d_t/\sqrt{v_t}+\lambda x_t\rangle$ is positive while running any particular algorithm. In order to do so, we approximate $x_*$ with the last iterate after running the algorithm for many iterations.
% and then repeat the run with the same random seed in initialization and data sampling. 
% \todo{TJ: in my experiments, I actually only ran each job once and saved many the checkpoints of the model through out the training, then checked aiming using the checkpoint data.}
Figure~\ref{fig:cosine_sgd} and~\ref{fig:cosine_bcos} depict the empirical aiming condition while running SGD, SGD with momentum, BCOS-c and BCOSW-c on GPT2-124M with hyper-parameters as chosen in Section~\ref{sec:experiments}.
% with $\alpha_{\max}=0.002, \beta=0.9, \lambda=0.1$ 
Specifically, we record the cosines between $x_t-x_\ast$ and various directions including stochastic gradients $g_t$, momentum $m_t$, $m_t/\sqrt{v_t}$ and $m_t/\sqrt{v_t} + \lambda x_t$ along the training trajectory. 
% \todo{TJ: shall we add $\epsilon$ to the demonimator?}
We notice that the cosines of the angles between $x_t-x_*$ and the search directions are positive (correct aiming) most of the time, especially in the later stage of training. Moreover, the actual update directions in each algorithm exhibit the highest cosine values with $x_t-x_\ast$ among all these directions.

The above results give empirical evidence that the aiming condition can be much weaker than convexity and often holds in training deep learning models. 
We explain the connection between aiming and convexity in more detail in Appendix~\ref{sec:aiming-vs-convexity}.
% \todo{Moved all subtlety about aiming vs convexity into Appendix.}

\begin{figure}[t]
    \centering
\ifloadfigs
    \input{figs/aiming.tex}
\else
    \includegraphics[width=0.99\textwidth]{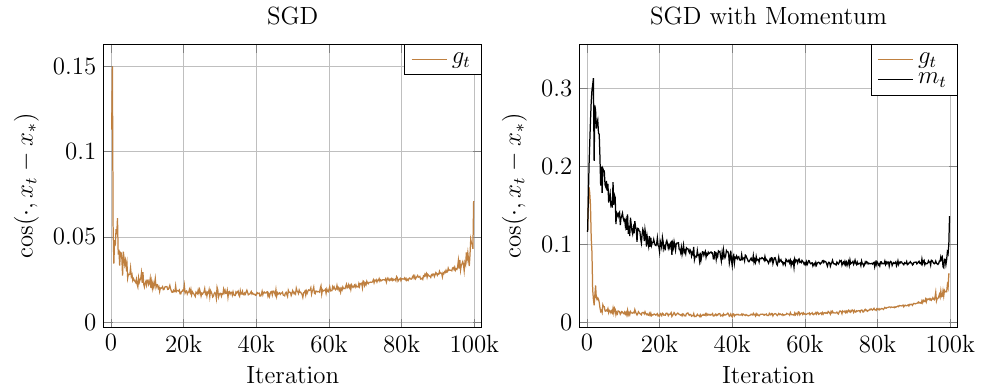}
\fi
\caption{Empirical aiming conditions on GPT2. Left: SGD. Right: SGD with momentum.
}
\label{fig:cosine_sgd}
\end{figure}

\begin{figure}[t]
    \centering
\ifloadfigs
    \input{figs/aiming.tex}
\else
    \includegraphics[width=0.99\textwidth]{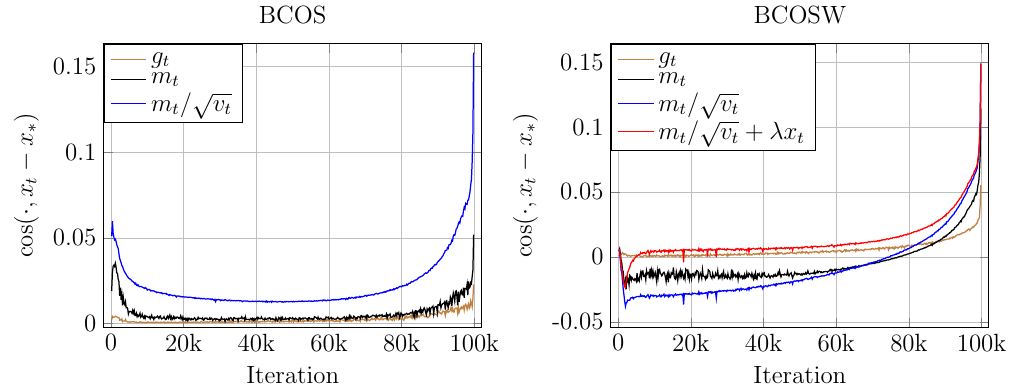}
\fi
\caption{Empirical aiming conditions on GPT2. Left: BCOS-c. Right: BCOSW-c.
}
\label{fig:cosine_bcos}
\end{figure}

% \todo{TJ: made text larger and added title. Shall we add $\epsilon$ to the denominator?}

\subsection{Analysis of conceptual BCOS}
\label{sec:analysis-conceptual}

In this section, we provide convergence analysis for the conceptual BCOS method in~\eqref{eqn:bcosw-conceptual}.
Our first result is about its one-step contraction property.
\begin{lemma}\label{lem:bcosw-one-step}
Suppose Assumption~\ref{assum:aiming} holds, $\alpha_t\geq 0$ and $\alpha_t\lambda<1$ for all $t\geq 0$. Then the sequence $\{x_t\}$ generated by~\eqref{eqn:bcosw-conceptual} satisfies, for all $t\geq 0$,
\begin{equation}\label{eqn:bcosw-contraction}
\E_t \bigl[\norm{x_{t+1}-x_\ast}^2\bigr] 
\leq (1-\alpha_t \lambda)^2\norm{x_{t}-x_\ast}^2 + \alpha_t^2 B_*,
\end{equation}
where 
\begin{equation}\label{eqn:c*}
B_\ast=n + \lambda^2 \norm{x_\ast}^2 + 2\lambda \norm{x_\ast}_1.
\end{equation}
\end{lemma}

\begin{proof}
First, we notice that the aiming condition~\eqref{eqn:aiming} is equivalent to
\begin{equation}\label{eqn:bcosw-aiming-x*}
\biggl\langle x_t-x_*, \, \frac{\E_t[d_t]}{\sqrt{\E_t[d_t^2]}} + \lambda x_* \biggr\rangle \geq 0.
\end{equation}
To see this, we simply add and subtract $\lambda x_t$ inside the inner product:
\begin{align*}
\biggl\langle x_t-x_*, \, \frac{\E_t[d_t]}{\sqrt{\E_t[d_t^2]}} + \lambda x_* \biggr\rangle 
&=
\biggl\langle x_t-x_*, \, \frac{\E_t[d_t]}{\sqrt{\E_t[d_t^2]}} + \lambda x_t - \lambda x_t + \lambda x_* \biggr\rangle \\
&=\biggl\langle x_t-x_*, \, \frac{\E_t[d_t]}{\sqrt{\E_t[d_t^2]}} + \lambda x_t \biggr\rangle - \lambda \|x_t - x_*\|^2 .
\end{align*}

Given $x_{t+1} = x_t - \widetilde\gamma_t\odot d_t - \alpha_t x_t$, we have
\begin{align}
\E_t \bigl[\|x_{t+1}-x_\ast\|^2\bigr] 
&=\E_t \!\left[\bigl\|x_{t} - \widetilde\gamma_t\odot d_t - \alpha_t \lambda x_t -x_\ast\bigr\|^2\right] \nonumber \\
&=\E_t \!\left[\bigl\|(1-\alpha_t \lambda)x_{t} - \widetilde\gamma_t\odot d_t -(1-\alpha_t \lambda)x_\ast - \alpha_t \lambda x_\ast\bigr\|^2\right] \nonumber \\
&=\E_t \!\left[\left\|(1-\alpha_t \lambda)(x_{t} - x_*) - \bigl( \widetilde\gamma_t\odot d_t + \alpha_t \lambda x_\ast\bigr) \right\|^2\right] \nonumber \\
&=\E_t \!\left[\bigl\|(1-\alpha_t \lambda)(x_{t} -x_\ast)\bigr\|^2\right] - 2 \E_t \!\left[\bigl\langle (1-\alpha_t \lambda)(x_{t} -x_\ast),\, \widetilde\gamma_t\odot d_t  + \alpha_t \lambda x_\ast\bigr\rangle\right] \nonumber \\
&\quad + \E_t \!\left[\bigl\|\widetilde\gamma_t\odot d_t  + \alpha_t \lambda x_\ast\bigr\|^2\right].
\label{eqn:bcosw-next-dist}
\end{align}
We examine the three terms in the last equation one by one.
Using the definition of $\E_t[\cdot]$, we can remove the conditional expectation from the first term, i.e.,
\[
\E_t \!\left[\bigl\|(1-\alpha_t \lambda)(x_{t} -x_\ast)\bigr\|^2\right] 
= \bigl\|(1-\alpha_t \lambda)(x_{t} -x_\ast)\bigr\|^2 
= (1-\alpha_t \lambda)^2 \bigl\|x_{t} -x_\ast\bigr\|^2 .
\]
For the second term, we use $\widetilde\gamma_t=\alpha_t/\sqrt{\E_t[d_t^2]}$ to obtain
\begin{align*}
\E_t \!\left[\bigl\langle (1-\alpha_t \lambda)(x_{t} -x_\ast),\, \widetilde\gamma_t\odot d_t  + \alpha_t \lambda x_\ast\bigr\rangle\right] 
&= \bigl\langle (1-\alpha_t \lambda)(x_{t} -x_\ast),\, \widetilde\gamma_t\odot \E_t[d_t] + \alpha_t \lambda x_\ast\bigr\rangle \\
&= \alpha_t(1-\alpha_t \lambda)\biggl\langle x_{t} -x_\ast,\, \frac{\E_t[d_t]}{\sqrt{\E_t[d_t^2]}} + \lambda x_\ast\biggl\rangle \\
&\geq 0,
\end{align*}
where the last inequality is due to~\eqref{eqn:bcosw-aiming-x*}.
For the third term in~\eqref{eqn:bcosw-next-dist}, we have
\allowdisplaybreaks
\begin{align*}
\E_t \!\left[\bigl\|\widetilde\gamma_t\odot d_t  + \alpha_t \lambda x_\ast\bigr\|^2\right]
&=\alpha_t^2\, \E_t \!\left[ \biggl\|\frac{d_t}{\sqrt{\E_t[d_t^2]}}+\lambda x_*\biggr\|^2\right] \\
&=\alpha_t^2\, \E_t \!\left[ \sum_{k=1}^n\left(\frac{d_{t,k}^2}{\E_t[d_{t,k}^2]} + \lambda^2 x_{*,k}^2 + 2\lambda x_{*,k} \frac{d_{t,k}}{\sqrt{\E_t[d_{t,k}^2]}}\right)\right] \\
&=\alpha_t^2\, \left[ \sum_{k=1}^n\left(\frac{\E_t[d_{t,k}^2]}{\E_t[d_{t,k}^2]} + \lambda^2 x_{*,k}^2 + 2\lambda x_{*,k} \frac{\E[d_{t,k}]}{\sqrt{\E_t[d_{t,k}^2]}}\right)\right] \\
&=\alpha_t^2 \sum_{k=1}^n \left(1+\lambda^2 x_{*,k}^2 + 2\lambda x_{*,k} \sqrt{\rho_{t,k}}\right) \\
&=\alpha_t^2 \bigl(n + \lambda^2\|x_*\|^2 + 2\lambda \langle x_*, \sqrt{\rho_t}\rangle\bigr) \\
&\leq \alpha_t^2 \bigl(n + \lambda^2\|x_*\|^2 + 2\lambda \| x_*\|_1 \|\sqrt{\rho_t}\|_\infty \bigr) \\
&\leq \alpha_t^2 \bigl(n + \lambda^2\|x_*\|^2 + 2\lambda \| x_*\|_1 \bigr), 
\end{align*}
where the last two inequalities are due to H\"older's inequality
$\langle x_*,\sqrt{\rho_t}\rangle\leq \|x_*\|_1\|\sqrt{\rho_t}\|_\infty$ and the fact that $\|\sqrt{\rho_t}\|_\infty\leq 1$, respectively.
Plugging the above equality and inequalities into~\eqref{eqn:bcosw-next-dist} yields the desired result.
\end{proof}

As a direct consequence of Lemma~\ref{lem:bcosw-one-step}, there exists sufficiently small $\alpha_t>0$ such that 
\[
\E_t \bigl[\norm{x_{t+1}\!-\!x_\ast}^2\bigr] \leq \|x_t\!-\!x_*\|^2, 
\qquad \forall\, t\geq 0.
\]
In fact, we can prove much stronger results, including almost sure convergence of~$x_t$ to~$x_*$ and characterizing the rate of convergence of $\E[\|x_t-x_*\|^2]$. 
These are presented in the following subsections.

\subsubsection{Almost sure convergence}

The following theorem establishes almost sure (a.s.) convergence of the conceptual BCOS algorithm under a classical condition on the sequence $\{\alpha_t\}$.

\begin{theorem}[Almost sure convergence]\label{thm:bcosw-conceptual-as}
Suppose the stepsize schedule $\{\alpha_t\}_{t\geq 0}$ satisfies
\begin{equation}\label{eqn:alpha_t-bcosw}
\alpha_t \geq 0,\quad  \alpha_t\lambda \leq 1,\quad \forall\, t\geq 0, \qquad 
\sum_{t=0}^\infty \alpha_t = \infty, \qquad \sum_{t=0}^\infty \alpha_t^2 < \infty.
\end{equation}
The under Assumption~\ref{assum:aiming}, the sequence $\{x_t\}$ generated by~\eqref{eqn:bcosw-conceptual} satisfies $\norm{x_t-x_\ast} \to 0$ a.s.
\end{theorem}

The proof of Theorem~\ref{thm:bcosw-conceptual-as} is based on the following \emph{almost supermartingale} lemma of \citet{ROBBINS1971}.
\begin{lemma}[{\citet[Theorem~1]{ROBBINS1971}}]
\label{lem:almost_supermartingale}
Let $(\Omega, \cF, P)$ be a probability space, and $\cF_0 \subset \cF_1 \subset \ldots$ be a sequence of sub-$\sigma$-algebras of $\cF$. For each $t\geq 0$, let $X_t$, $a_t$, $b_t$ and $c_t$ be non-negative $\cF_t$-measurable random variables such that
\begin{equation}\label{eq:almost_supermartingale_recursive}
    \E[X_{t+1} \vert \cF_t] \leq X_t(1+a_t) + b_t - c_t.
\end{equation}
If $\sum_{t=0}^\infty a_t < \infty$ and $\sum_{t=0}^\infty b_t < \infty$, then $\lim_{t \to \infty} X_t$ exists and is finite, and $\sum_{t=0}^\infty c_t < \infty$ almost surely (a.s.).
\end{lemma}
    
\begin{proof}[Proof of Theorem~\ref{thm:bcosw-conceptual-as}]
Define $X_t:=\norm{x_{t}-x_\ast}^2$ and $\cF_t$ to be the $\sigma$-algebra generated by $X_0, \cdots, X_t$. 
Lemma~\ref{lem:bcosw-one-step} implies the following recursion,
in the form of~\eqref{eq:almost_supermartingale_recursive},
\begin{align*}
\E[X_{t+1}\vert\cF_t] 
&\leq (1 - \alpha_t \lambda)^2 X_t + \alpha_t^2 B_\ast \\
&= (1 + \alpha_t^2 \lambda^2) X_t + \alpha_t^2 B_\ast  - 2 \alpha_t \lambda X_t\\
&= \left(1 + a_t \right)X_t + b_t - c_t,
\end{align*}
where $a_t= \alpha_t^2 \lambda^2$, $b_t=\alpha_t^2 B_\ast$ and $c_t=2 \alpha_t \lambda X_t$. 
Notice that $X_t$, $a_t$, $b_t$ and $c_t$ are all non-negative, and the square-summable assumption of $\alpha_t $ guarantees that
\[
\sum_{t=0}^\infty a_t = \sum_{t=0}^\infty \alpha_t^2 \lambda^2 < \infty, \qquad \sum_{t=0}^\infty b_t = \sum_{t=0}^\infty \alpha_t^2 B_\ast < \infty.
\]
Therefore, all the assumptions in Lemma~\ref{lem:almost_supermartingale} are satisfied. We may now apply the lemma to conclude that 
$X_t=\norm{x_{t}-x_\ast}^2 \to X$ for some $X < \infty$ and
% \todo{Tao: could you check the technicality about a.s. convergence you raised in email early?}
\[
\sum_{t=0}^\infty c_t= \sum_{t=0}^\infty 2 \alpha_t \lambda \norm{x_{t}-x_\ast}^2 < \infty \quad \text{ a.s.}
\]
This is compatible with the assumption $\sum_{t=0}^\infty \alpha_t = \infty$ only if $\norm{x_{t}-x_*}^2 \to 0$ a.s.
\end{proof}

\subsubsection{Rates of convergence}
First, with a constant stepsize schedule, we have linear convergence to a neighborhood of~$x_*$ as stated by the following corollary.
\begin{corollary}\label{cor:bcosw_constant}
Consider the conceptual BCOS method~\eqref{eqn:bcosw-conceptual} with a constant step size schedule $\alpha_t = \alpha$ where $\alpha$ satisfies $\alpha\lambda<1$. 
Then under Assumption~\ref{assum:aiming}, we have
\[
\E\bigl[\norm{x_{t}-x_\ast}^2\bigr] \leq (1-\alpha \lambda)^{2t}\,\E\bigl[\norm{x_{0}-x_\ast}^2\bigr] + \frac{\alpha^2 B_\ast}{1-(1-\alpha\lambda)^{2}},
\]
where $B_\ast$ is defined in~\eqref{eqn:c*}.
\end{corollary}

\begin{proof}
Recursive application of Lemma~\ref{lem:bcosw-one-step} with $\alpha_t=\alpha$ and taking total expectation yields 
% \todo{TJ: added taking total expectation}
\begin{align*}
    \E\bigl[\norm{x_{t}-x_\ast}^2\bigr] &\leq (1-\alpha \lambda)^{2t}\, \E\bigl[\norm{x_{0}-x_\ast}^2\bigr] + \sum_{s=0}^{t-1} (1-\alpha\lambda)^{2(t-s-1)}\alpha^2 B_\ast\\
    &= (1-\alpha \lambda)^{2t}\, \E\left[\norm{x_{0}-x_\ast}^2\right] + \sum_{s=0}^{t-1} (1-\alpha\lambda)^{2s}\alpha^2 B_\ast\\
    &\leq (1-\alpha \lambda)^{2t}\, \E\left[\norm{x_{0}-x_\ast}^2\right] + \frac{\alpha^2 B_\ast}{1-(1-\alpha\lambda)^{2}},
\end{align*}
which is the desired result.
\end{proof}

\smallskip

With a diminishing stepsize schedule $\alpha_t$, we can show that $\E[\|x_t-x_*\|^2]\to 0$ at a sublinear rate, as given in the following theorem.

\begin{theorem}\label{thm:bcosw_1/t}
Consider the algorithm~\eqref{eqn:bcosw-conceptual} with stepsize schedule $\alpha_t = \frac{\alpha}{t+1}$.
Suppose the constant $\alpha$ satisfies $1/2<\alpha \lambda<1$ and Assumption~\ref{assum:aiming} holds.
Then we have for all $t\geq 0$,
\[
\E [\norm{x_{t}-x_\ast}^2] \leq \frac{\alpha^2 \left(\lambda^2 \E\left[\norm{x_{0}-x_\ast}^2\right] + (1+\pi^2/6) B_\ast \right)}{2\alpha \lambda - 1} \frac{1}{t+1} + \cO\left(\frac{1}{(t+1)^2}+\frac{1}{(t+1)^{2\alpha \lambda}}\right),
\]
where $B_\ast$ is defined in~\eqref{eqn:c*}.
\end{theorem}

Our proof of Theorem~\ref{thm:bcosw_1/t} is based on Lemma~\ref{lem:bcosw-one-step} and a classical result of \citet{chung1954}.

\begin{lemma}[{\citet[Lemma~1]{chung1954}}]
\label{lem:chung1}
    Suppose that $\{X_t\}$ is a sequence of real numbers such that for all $t\geq 1$, 
    \begin{equation}\label{eqn:chung1-recursion}
        X_{t+1} \leq \left(1 - \frac{a}{t}\right) X_t + \frac{b}{t^{p+1}},
    \end{equation}
where $a > p > 0$ and $b > 0$. Then
    \[
    X_t \leq \frac{b}{a-p} \frac{1}{t^p} + \mathcal O \left(\frac{1}{t^{p+1}} + \frac{1}{t^a}\right).
    \]
\end{lemma}

In order to apply Chung's lemma, we need to first derive an upper bound on $\E[\|x_t-x_*\|^2]$ for any $t\geq 0$, which will be incorporated into the constant~$b$ in the lemma.
To this end, we start with the inequality~\eqref{eqn:bcosw-contraction} and use the law of total expectation to get
\begin{align}
\E\bigl[\norm{x_{t}-x_\ast}^2\bigr] 
&= \E\!\left[\E_{t-1} [\norm{x_{t}-x_\ast}^2\bigr]\right] \nonumber \\
&\leq \E\!\left[(1-\alpha_{t-1}\lambda)^2 \|x_{t-1}-x_*\|^2 + \alpha_{t-1}^2 B_\ast\right] \nonumber \\
&\leq \E\!\left[ \|x_{t-1}-x_*\|^2\right] + \alpha_{t-1}^2 B_\ast \nonumber \\
&\leq \E\bigl[ \|x_{0}-x_*\|^2\bigr] +  B_\ast\sum_{s=0}^{t-1}\alpha_s^2,
\label{eqn:dist-upper-bound}
\end{align}
where the last term is bounded if $\sum_{t=1}^\infty\alpha_t^2<\infty$. 

\bigskip

\begin{proof}[Proof of Theorem~\ref{thm:bcosw_1/t}]
Taking total expectation on both sides of~\eqref{eqn:bcosw-contraction}, we obtain
\begin{align}
\E \bigl[\norm{x_{t+1}-x_\ast}^2\bigr] 
&\leq (1-\alpha_t \lambda)^2\E\bigl[\norm{x_{t}-x_\ast}^2\bigr] + \alpha_t^2 B_\ast  \nonumber \\
&= (1-2\alpha_t \lambda)\E\bigl[\norm{x_{t}-x_\ast}^2\bigr] + \alpha_t^2 \left(\lambda^2\E\bigl[\|x_t-x_*\|^2\bigr]+B_\ast\right) .
\label{eqn:dist-expand-square}
\end{align}
We bound~$\E[\|x_t-x_*\|^2]$ by substituting $\alpha_t=\frac{\alpha}{t+1}$ into~\eqref{eqn:dist-upper-bound}, which leads to
\begin{align*}
\E\bigl[\norm{x_{t}-x_\ast}^2\bigr] 
& \leq \E\bigl[ \|x_{0}-x_*\|^2\bigr] +  B_\ast\sum_{s=0}^{\infty}\frac{\alpha^2}{(s+1)^2} \\
&= \E\bigl[ \|x_{0}-x_*\|^2\bigr] +  \frac{\pi^2}{6} \alpha^2 B_\ast.
\end{align*}
Applying the above bound to the last term of~\eqref{eqn:dist-expand-square} and setting $\alpha_t=\frac{\alpha}{t+1}$, we get
\begin{align*}
\E \bigl[\norm{x_{t+1}-x_\ast}^2\bigr] 
&\leq \left(1-\frac{2\alpha \lambda}{t+1}\right)\E\bigl[\norm{x_{t}-x_\ast}^2\bigr] + \frac{\alpha^2\!\left(\lambda^2\E\bigl[\|x_0-x_*\|^2\bigr] +(\pi^2/6)\alpha^2\lambda^2 B_\ast + B_\ast\right)}{(t+1)^2} \\
&\leq \left(1-\frac{2\alpha \lambda}{t+1}\right)\E\bigl[\norm{x_{t}-x_\ast}^2\bigr] + \frac{\alpha^2\!\left(\lambda^2\E\bigl[\|x_0-x_*\|^2\bigr] +(1+\pi^2/6) B_\ast\right)}{(t+1)^2} ,
\end{align*}
where in the second inequality we used the assumption $\alpha\lambda<1$.

Let $X_{t+1}=\E[\|x_t-x_*\|^2]$, $a=2\alpha\lambda$, $p=1$ and $b=\alpha^2\left(\lambda^2\E[\|x_0-x_*\|^2]+(1+\pi^2/6)B_\ast\right)$. 
Then the above inequality can be written as
\[
X_{t+2} \leq \left(1-\frac{a}{t+1}\right) X_{t+1} + \frac{b}{(t+1)^{p+1}},
\]
which is in the form of~\eqref{eqn:chung1-recursion} except with the index shift $t\to t+1$.
Applying Lemma~\ref{lem:chung1} gives the desired result.
\end{proof}

\smallskip

Theorem~\ref{thm:bcosw_1/t} requires $\alpha\lambda>1/2$ to obtain the $\cO(1/t)$ rate of convergence. We can remove this condition by working with the stepsize schedule $\alpha_t=\frac{\alpha}{(t+1)^p}$ with $p\in(\frac{1}{2},1)$. 
%The following theorem states that in this case we have $\E[\|x_t-x_*\|^2]=\cO(1/t^p)$.
In this case we have an \emph{asymptotic} $\cO(1/t^p)$ rate, as stated in the following theorem.

\begin{theorem}\label{thm:bcosw_1/t^p}
Consider the algorithm~\eqref{eqn:bcosw-conceptual} with stepsize schedule $\alpha_t = \frac{\alpha}{(t+1)^p}$ where $p \in \left(\frac{1}{2},1\right)$ and the constant $\alpha$ satisfies $\alpha \lambda < 1$. 
Then under Assumption~\ref{assum:aiming}, we have
\begin{equation}\label{eqn:bcosw_1/t^p}
    \limsup_{t\to \infty}\, (t+1)^{p}\,\E \bigl[\norm{x_{t}-x_\ast}^2\bigr] \leq \frac{\alpha \left(\lambda^2 \E\bigl[\norm{x_{0}-x_\ast}^2\bigr] + (1+K_p) B_\ast \right)}{2\lambda},
\end{equation}
where $K_p=\sum_{t=1}^\infty \frac{1}{(t+1)^{2p}}$ is finite for $p\in(\frac{1}{2},1)$.
\end{theorem}
The proof is based on another lemma of \citet{chung1954}, restated below.

\begin{lemma}[{\citet[Lemma 4]{chung1954}}]
\label{lem:chung_lem4}
    Suppose that $\{X_t\}$ is a sequence of real numbers such that for all $t\geq 1$, 
    \begin{equation}\label{eq:chung_inequality_cpt3_lem4}
        X_{t+1} \leq \left(1 - \frac{a_t}{t^p}\right) X_t + \frac{b}{t^{q}},
    \end{equation}
    where $0< p < 1, p<q, a_t \geq a > 0, b>0$. Then
    \[
    \limsup_{t \to \infty}\, t^{q-p} X_t \leq \frac{b}{a}.
    \]
\end{lemma}

\begin{proof}[Proof of Theorem~\ref{thm:bcosw_1/t^p}]
Substituting $\alpha_t=\frac{\alpha}{(t+1)^p}$ into~\eqref{eqn:dist-upper-bound}, we obtain
\[
\E\bigl[\|x_t-x_*\|^2\bigr] 
\leq \E\bigl[\|x_0-x_*\|^2\bigr] 
+ B_\ast\sum_{s=0}^\infty \frac{\alpha^2}{(t+1)^{2p}}
= \E\bigl[\|x_0-x_*\|^2\bigr] + K_p\alpha^2 B_\ast.
\]
Taking total expectation of both sides of~\eqref{eqn:bcosw-contraction} and plugging in $\alpha_t=\frac{\alpha}{(t+1)^p}$, we get
\begin{align*}
\E \bigl[\norm{x_{t+1}-x_\ast}^2\bigr] 
        &\leq (1-\alpha_{t} \lambda)^2 \E\bigl[\norm{x_{t}-x_\ast}^2\bigr] + \alpha_t^2 B_\ast\\
        &= \left(1-2\alpha_t \lambda\right) \E[\norm{x_{t}-x_\ast}^2] + \alpha_t^2\left(B_\ast + \lambda^2 \E\bigl[\norm{x_{t}-x_\ast}^2\bigr]\right)\\
        &\leq \left(1-\frac{2\alpha \lambda}{(t+1)^p}\right) \E[\norm{x_{t}-x_\ast}^2] + \frac{\alpha^2 \left(\lambda^2 \E\bigl[\norm{x_{0}-x_\ast}^2\bigr]+(1+K_p)B_\ast \right)}{(t+1)^{2p}}.
\end{align*}
The desired result follows by applying Lemma~\ref{lem:chung_lem4} with the definitions
$X_{t+1}=\E[\norm{x_{t}-x_\ast}^2]$, $a_t=a=2\alpha \lambda$, $b=\alpha^2 \left(\lambda^2 \E[\norm{x_{0}-x_\ast}^2] + (1+K_p)B_\ast\right)$ and $q=2p$.
\end{proof}

\subsubsection{Convergence analysis with $\lambda=0$}
%\subsubsection{Analysis for not using decoupled weight decay}

Here we briefly explain the convergence analysis of BCOS in the case of $\lambda=0$.
%or in the case of $\lambda>0$ but without using decoupled weight decay.
In other words, we consider
\begin{equation}\label{eqn:bcos-no-dwd}
x_{t+1} = x_t - \widetilde\gamma_t\odot d_t = x_t - \alpha_t\frac{d_t}{\sqrt{\E_t[d_t^2]}}.
\end{equation}
The aiming condition~\eqref{eqn:aiming} becomes
\begin{equation}\label{eqn:aiming-wd0}
\biggl\langle x_t-x_*,\,\frac{\E_t[d_t]}{\sqrt{\E_t[d_t^2]}}\biggr\rangle \geq 0 .
\end{equation}
In order to establish similar convergence guarantees as before, we need slightly stronger aiming conditions, as discussed below. 
\begin{enumerate}
\item
First, with $\lambda=0$, we have the following inequality in place of~\eqref{eqn:bcosw-contraction} (see Lemma~\ref{lem:bcosw-one-step}):
\[
\E_t\bigl[\|x_{t+1}-x_*\|^2\bigr] \leq \|x_t-x_*\|^2 - \alpha_t \biggl\langle x_t-x_*,\,\frac{\E_t[d_t]}{\sqrt{\E_t[d_t^2]}}\biggr\rangle + \alpha_t^2 n.
\]
To have contraction with sufficiently small~$\alpha_t$, we need the aiming condition~\eqref{eqn:aiming-wd0} to hold with strict inequality for any $x_t\neq x_*$.
\item 
Almost sure convergence as in Theorem~\ref{thm:bcosw-almost-sure} requires a slightly stronger condition
\[
\biggl\langle x_t-x_*,\,\frac{\E_t[d_t]}{\sqrt{\E_t[d_t^2]}}\biggr\rangle \geq \phi(\|x_t-x_*\|),
\]
where $\phi$ is any strictly positive and continuous function.
\item 
To obtain sublinear rates of convergence as in Theorems~\ref{thm:bcosw_1/t} and~\ref{thm:bcosw_1/t^p}, we need an even stronger aiming condition, specifically,
\begin{equation}\label{eqn:aiming-mu}
\biggl\langle x_t-x_*,\,\frac{\E_t[d_t]}{\sqrt{\E_t[d_t^2]}}\biggr\rangle \geq \mu \|x_t-x_*\|^2,
\end{equation}
with some $\mu>0$ for all $t\geq 0$.
\end{enumerate}

\subsection{Analysis for practical BCOS}
\label{sec:analysis-practical}

In this section, we provide convergence analysis for the practical BCOS method
\begin{equation}\label{eq:bcosw-practical-d}
    x_{t+1} = (1-\alpha_t\lambda) x_t - \alpha_t \frac{d_t}{\sqrt{v_t+\epsilon}},
\end{equation}
where $v_t$ is an online estimator of the conditional second moment $\E_t[d_t^2]$.
As explained at the beginning of Section~\ref{sec:convergence}, our analysis is based on bounding the difference between the expected steps under the practical method and the conceptual method, in the form of~\eqref{eqn:step-diff}.
To proceed, we first decompose their difference into two terms, i.e., 
\begin{equation}\label{eqn:difference-bound-in-two}
    \left|\frac{\E_t[d_t ]}{\sqrt{\E[d_t^2]}} - \E_t\biggl[\frac{d_t}{\sqrt{v_t \!+\!\epsilon}}\biggr]\right| \leq \left| \frac{\E_t[d_t ]}{\sqrt{\E[d_t^2]}} - \frac{\E_t[d_t]}{\sqrt{\E_t[v_t]\!+\!\epsilon}}\right| + \left|\frac{\E_t[d_t]}{\sqrt{\E_t[v_t]\!+\!\epsilon}} - \E_t\biggl[\frac{d_t}{\sqrt{v_t \!+\!\epsilon}} \biggr]\right|,
\end{equation}
and then bound the two error terms on the right-hand side separately. Intuitively, the magnitudes of these two error terms are determined by the quality of the estimator~$v_t$.
We make the following assumption on the bias of $v_t$.

\begin{assumption}[Bias of second-moment estimator]
\label{assum:v_t-bias}
There exists $\tau\in(0,1)$ and $\epsilon>0$ such that for all $t\geq 0$,
\begin{equation}\label{eqn:estimator-bias}
\left |\E_t[v_t] - \E_t [d_t^2] \right| \leq \tau \E_t[d_t^2] +\epsilon.
\end{equation}
\end{assumption}

Here $\epsilon$ is an additive error bound that leaves room for $\E_t[v_t]$ when $\E_t[d_t^2]$ is very small.
There is no loss of generality by restricting $\tau<1$ in the sense that we can always find an estimator that satisfies the assumption with sufficiently large~$\epsilon$.
As an extreme case, by setting $\epsilon\geq\|\E_t[d_t^2]\|_\infty$ for all~$t$, the trivial estimator $v_t=0$ satisfies~\eqref{eqn:estimator-bias} for any $\tau\in(0,1)$.

\subsubsection{Bounding the difference between conceptual and practical updates}

The following lemma derives an upper bound on the first term on the right-hand side of~\eqref{eqn:difference-bound-in-two}.
The proof is given in Appendix~\ref{sec:appendix-frac-expect}.

\begin{lemma}\label{lem:bound-frac-expect}
Under Assumption~\ref{assum:v_t-bias}, it holds that:
\begin{equation}
\left| \frac{\E_t[d_t]}{\sqrt{\E_t[d_t^2]}} - \frac{\E_t[d_t]}{\sqrt{\E_t[v_t]+\epsilon}} \right| 
%\leq \frac{4\tau+3\tau^2+\cO(\tau^3)}{8}
\leq \frac{\tau+\cO(\tau^2)}{2}
\left|\frac{\E_t[d_t]} {\sqrt{\E_t[d_t^2]}}\right| + \cO(\epsilon).
 \end{equation}
\end{lemma}

\bigskip

To bound the second term on the right-hand side of~\eqref{eqn:difference-bound-in-two}, we need the following lemma, whose proof is given in Appendix~\ref{sec:proof-lemma-ratio-sqrt}.
% \todo{$\frac{Y}{\sqrt{Z}+\epsilon}$, actually, this may not look good.}
\begin{lemma}\label{lem:expectation_ratio_sqrt}
   Let $Y,Z$ be two random variables and $Z>0$ almost surely, then 
\begin{equation}\label{eq:lem:expectation_ratio_sqrt}
\begin{aligned}
    \E \left[\frac{Y}{\sqrt{Z}}\right] =& \frac{\E[Y]}{\sqrt{\E[Z]}} \left(1 - \frac{\cov(Y,Z)}{2\E[Y]\E[Z]} + \frac{3\var(Z)}{8\E[Z]^2}
+ \cO\left(\frac{\Var(Z)}{\E[Z]^2}\right)
\right) 
\end{aligned}
\end{equation}
where $\cO\left(\frac{\Var(Z)}{\E[Z]^2}\right)$ includes higher-order statistical terms
\[
\frac{\E\bigl[(Y-\E[Y])(Z-\E[Z])^{p-1}\bigr]}{\E[Y]\E[Z]^{p-1}}
\qquad\text{and}\qquad
\frac{\E\bigl[(Z-\E[Z])^p\bigr]}{\E[Z]^{p}},
\]
for $p\geq 3$.
\end{lemma}
\iffalse
\todo{An alternative version of Lemma 5.6 with expansion on function $\frac{Y}{\sqrt{Z}+\epsilon}$}
\begin{lemma}\label{lem:expectation_ratio_sqrt}
   Let $Y,Z$ be two random variables and $Z>0$ almost surely, then 
\begin{equation}\label{eq:lem:expectation_ratio_sqrt}
\begin{aligned}
    \E \left[\frac{Y}{\sqrt{Z+\epsilon}}\right] =& \frac{\E[Y]}{\sqrt{\E[Z]+\epsilon}} \left(1 - \frac{\cov(Y,Z)}{2\E[Y](\E[Z]+\epsilon)} + \frac{3\var(Z)}{8(\E[Z]+\epsilon)^2}
+ \cO\left(\frac{\Var(Z)}{(\E[Z]+\epsilon)^2}\right)
\right) 
\end{aligned}
\end{equation}
where $\cO\left(\frac{\Var(Z)}{(\E[Z]+\epsilon)^2}\right)$ includes higher-order statistical terms
\[
\frac{\E\bigl[(Y-\E[Y])(Z-\E[Z])^{p-1}\bigr]}{\E[Y](\E[Z]+\epsilon)^{p-1}}
\qquad\text{and}\qquad
\frac{\E\bigl[(Z-\E[Z])^p\bigr]}{(\E[Z]+\epsilon)^{p}},
\]
for $p\geq 3$.
\end{lemma}
\fi
Applying Lemma~\ref{lem:expectation_ratio_sqrt} with $Y=d_t$ and $Z=v_t+\epsilon$, we get the coordinate-wise inequality
\begin{align}
\left|\frac{\E_t[d_t]}{\sqrt{\E_t[v_t]\!+\!\epsilon}} - \E_t\!\left[\frac{d_t}{\sqrt{v_t\!+\!\epsilon}} \right]\right| 
\leq \left|\frac{\E_t[d_t]}{\sqrt{\E_t[v_t]\!+\!\epsilon}}\right|\odot\left|\frac{-\cov_t(d_t, v_t\!+\!\epsilon)}{2\E_t[d_t]\E_t[v_t\!+\!\epsilon]} 
+ \cO\!\left(\frac{\Var_t(v_t\!+\!\epsilon)}{\E_t[v_t\!+\!\epsilon]^2}\right) \right|.
\label{eqn:2nd-error-bound-interm}
\end{align}
In order to gain more intuition, we express the above bound using the conditional \emph{signal-to-noise ratio} (SNR) of $d_t$ and $v_t+\epsilon$, defined as
\[
\SNR_t(d_t)=\frac{\E_t[d_t]^2}{\Var_t(d_t)}, \qquad
\SNR_t(v_t+\epsilon)=\frac{\E_t[v_t+\epsilon]^2}{\Var_t(v_t+\epsilon)} =\frac{\E_t[v_t+\epsilon]^2}{\Var_t(v_t)}.
\]
With the definition of SNRs, the covariance term in~\eqref{eqn:2nd-error-bound-interm} can be expressed as
\begin{align*}
\frac{\cov_t(d_t, v_t+\epsilon)}{\E_t[d_t]\E_t[v_t+\epsilon]}
&=\frac{\cov_t(d_t, v_t)}{\E_t[d_t]\E_t[v_t+\epsilon]}\\
&=  \frac{\cov_t(d_t, v_t)}{\sqrt{\var_t(d_t)\var_t(v_t)}}\sqrt{\frac{\var_t(d_t)\var_t(v_t)}{\E_t[d_t]^2\E_t[v_t+\epsilon]^2}}\\
    &=\corr_t(d_t, v_t)\frac{1}{\sqrt{\SNR_t(d_t)\SNR_t(v_t+\epsilon)}},
\end{align*} 
where we use the definition of correlation of two random variables: $\corr(Y,Z)=\frac{\cov(Y,Z)}{\sqrt{\Var(Y)\Var(Z)}}$.
Substituting the expressions using SNRs into~\eqref{eqn:2nd-error-bound-interm}, we obtain
\begin{align*}
\left|\frac{\E_t[d_t]}{\sqrt{\E_t[v_t]\!+\!\epsilon}} - \E_t\!\left[\frac{d_t}{\sqrt{v_t\!+\!\epsilon}} \right]\right| 
  \leq \left|\frac{\E_t[d_t]}{\sqrt{\E_t[v_t]\!+\!\epsilon}}\right|\odot \left|\frac{-(1/2)\corr_t(d_t,v_t)}{\sqrt{\SNR_t(d_t)\SNR_t(v_t\!+\!\epsilon)}} 
+ \cO\!\left(\frac{1}{\SNR_t(v_t\!+\!\epsilon)} \right)\right|.
\end{align*}

\bigskip

We are ready to bound the difference between the expected updates of the practical BCOS and the conceptual BCOS methods. 
Specifically, bounding the first term on the right-hand side of~\eqref{eqn:difference-bound-in-two} using Lemma~\ref{lem:bound-frac-expect} and the second term using the last inequality above, we obtain
%\allowdisplaybreaks
\begin{align*}
\left|\frac{\E_t[d_t]}{\sqrt{\E_t[d_t^2]}} - \E_t\left[\frac{d_t}{\sqrt{v_t+\epsilon}}\right]\right| 
&\leq \frac{\tau+\cO(\tau^2)}{2}\left|\frac{\E_t[d_t]}{\sqrt{\E_t[d_t^2]}}\right| +\cO(\epsilon) \\
&\quad + \left|\frac{\E_t[d_t]}{\sqrt{\E_t[v_t]+\epsilon}}\right|\odot \left|\frac{-(1/2)\corr_t(d_t,v_t)}{\sqrt{\SNR_t(d_t)\SNR_t(v_t+\epsilon)}} 
+ \cO\!\left(\frac{1}{\SNR_t(v_t+\epsilon)} \right)\right|.
\end{align*}
We can further merge the two terms on the right-hand side by bounding $\left|\frac{\E_t[d_t]}{\sqrt{\E_t[v_t]+\epsilon}}\right|$ in terms of $\left|\frac{\E_t[d_t]}{\sqrt{\E_t[d_t^2]}}\right|$.
To this end, we use Lemma~\ref{lem:bound-frac-expect} to obtain
\[
\left|\frac{\E_t[d_t]}{\sqrt{\E_t[v_t]+\epsilon}}\right|
\leq \left(1+\frac{\tau+\cO(\tau^2)}{2}\right)
\left|\frac{\E_t[d_t]}{\sqrt{\E_t[d_t^2]}}\right| + \cO(\epsilon).
\]
Combining the two inequalities above, we arrive at the following result.

\begin{lemma}\label{lem:practicalBCOS_error_all}
Under Assumptions~\ref{assum:v_t-bias}, we have for all $t\geq 0$,
\begin{align*}
\left|\E_t\!\left[\frac{d_t}{\sqrt{v_t+\epsilon}}\right] - \frac{\E_t[d_t]}{\sqrt{\E_t[d_t^2]}}\right|
&\leq \sigma_t\left|\frac{\E_t[d_t]}{\sqrt{\E_t[d_t^2]}}\right| 
 + \cO(\epsilon),
\end{align*}
where 
% \todo{Need to change notation $c_t$ to $\theta_t$? to avoid clash with $c_t$ in Chung's lemmas and Dvoretzky's theorem.}
\begin{equation}\label{eqn:c_t-def}
\sigma_t=\frac{\tau \!+\! \cO(\tau^2)}{2} + \left(1+\frac{\tau \!+\! \cO(\tau^2)}{2}\right)\left\|\frac{-(1/2)\corr_t(d_t, v_t\!)}{\sqrt{\SNR_t(d_t)\SNR_t(v_t\!+\!\epsilon)}} + \cO\!\left(\frac{1}{\SNR_t(v_t\!+\!\epsilon)}\right) \right\|_\infty \!\!.
\end{equation}
\end{lemma}
% \todo{$\omega$ is used in Dvoretzky's theorem. Ops, then $\sigma_t$ looks good.}

\subsubsection{Almost sure convergence of practical BCOS}

%We are ready to prove almost sure convergence of the practical BCOS method.
The following lemma is a variant of the aiming condition with the practical BCOS update.

\begin{lemma}\label{lem:practical-aiming}
Under Assumptions~\ref{assum:v_t-bias}, it holds that
\begin{align}
   \left \langle x_t-x_\ast, \E_t \!\left[\frac{d_t}{\sqrt{v_t + \epsilon}} \right]\right \rangle \geq 
   &\left \langle x_t-x_\ast, \frac{\E_t[d_t]}{\sqrt{\E_t[d_t^2]}}\right \rangle - \norm{x_t-x_\ast}_1 \bigl( \sigma_t + \cO(\epsilon)\bigr),
\label{eqn:practical-aiming}
\end{align}
where $\sigma_t$ is given by~\eqref{eqn:c_t-def}.
\end{lemma}
   
\begin{proof}
   Adding and subtracting the term $\frac{\E_t[d_t]}{\sqrt{\E_t[d_t^2]}}$ from the inner product, we obtain:
\allowdisplaybreaks
\begin{align*}
    &\left \langle x_t-x_\ast, \E_t \!\left[\frac{d_t}{\sqrt{v_t + \epsilon}} \right]\right \rangle \nonumber\\
    =~& \left \langle x_t-x_\ast, \frac{\E_t[d_t]}{\sqrt{\E_t[d_t^2]}} \right \rangle + \left \langle x_t-x_\ast, \E_t \!\left[\frac{d_t}{\sqrt{v_t + \epsilon}} \right] -  \frac{\E_t[d_t]}{\sqrt{\E_t[d_t^2]}}\right \rangle \nonumber\\
    \geq~& \left \langle x_t-x_\ast,  \frac{\E_t[d_t]}{\sqrt{\E_t[d_t^2]}}\right \rangle - \norm{x_t-x_\ast}_1\cdot \norm{\E_t\!\left[\frac{d_t}{\sqrt{v_t + \epsilon}}\right] -  \frac{\E_t[d_t]}{\sqrt{\E_t[d_t^2]}}}_\infty \nonumber\\
    \geq~& \left \langle x_t-x_\ast, \frac{\E_t[d_t]}{\sqrt{\E_t[d_t^2]}}\right \rangle - \norm{x_t-x_\ast}_1 \left(\sigma_t \norm{\frac{\E_t[d_t]}{\sqrt{\E_t[d_t^2]}}}_\infty + \cO(\epsilon)\right),
%    + \cO\!\left(\epsilon,\, \frac{1}{\SNR_t(d_t)},\, \frac{1}{\SNR_t(v_t+\epsilon)}\right)\right),
%\label{lem:eq:inner_prod_lbd1}
\end{align*}
where the first inequality is due to H\"older's inequality 
and the last inequality is due to Lemma~\ref{lem:practicalBCOS_error_all}.
% \todo{We use H\"older's inequality here. With Cauchy-Schwarz 2-norm, need to add factor $\sqrt{n}$ to error terms $\sigma_t$ and $\cO(\epsilon)$.}
In addition, we notice that 
\begin{equation}\label{eqn:rho_t-inf-norm-bounded}
\biggl\|\frac{\E_t[d_t]}{\sqrt{\E_t[d_t^2]}}\biggr\|_\infty
=\left\|\sqrt{\rho_t}\right\|_\infty\leq 1,
\end{equation}
which concludes the proof.
\end{proof}

\begin{theorem}[Almost sure convergence of practical BCOS]
\label{thm:bcosw-almost-sure}
Consider the practical BCOS method~\eqref{eqn:bcosw-practical}.
Suppose that Assumptions~\ref{assum:aiming} and~\ref{assum:v_t-bias} hold, $d_t$ is bounded almost surely, and $\{\alpha_t\}$ satisfies~\eqref{eqn:alpha_t-bcosw}. 
Then there exists $\delta>0$ such that we have with probability one,
\begin{equation}\label{eqn:limsup-delta}
\limsup_{t\to\infty} \|x_t-x_*\|^2 \leq \delta^2.
\end{equation}
More specifically, $\delta$ is the smallest constant such that $\|x-x_*\|\geq \delta$ implies
% \todo{Here intentionally use $\|x-x_*\|$, not $\|x_t-x_*\|$!}
\begin{equation}\label{eqn:neighberhood-cond}
2\|x-x_*\|_1\bigl(\sigma_t+\cO(\epsilon)\bigr) \leq \lambda\|x-x_*\|^2, \qquad \forall\, t\geq 0,
\end{equation}
where $\sigma_t$ is given by~\eqref{eqn:c_t-def}.
%
%More specifically, let $\delta_t$ be the smallest constant such that $\|x_t-x_*\|\geq \delta_t$ implies
%\begin{equation}\label{eqn:neighberhood-cond}
%2\|x_t-x_*\|_1\bigl(c_t+\cO(\epsilon)\bigr) \leq \lambda\|x_t-x_*\|^2, \qquad \forall\, t\geq 0,
%\end{equatiok}
%where $c_t$ is given by~\eqref{eqn:c_t-def}.
%Then we have $\delta=\limsup_{t\to\infty}\delta_t$.
\end{theorem}

The proof of Theorem~\ref{thm:bcosw-almost-sure} is based on a classical result of stochastic approximation due to \citet{Dvoretzky1956} and its extension by \citet{DermanSacks1959}. Here we present a version that best fits our purpose.

\begin{theorem}[An extension of Dvoretzky's Theorem]
\label{thm:Dvoretzky}
Let $(\Omega=\{\omega\},\cF,P)$ be a probability space. 
Let $\{x_t\}$ and $\{y_t\}$ be sequences of random variables such that, for all $t\geq 0$,
% \todo{$\omega$ is used in Dvoretzky's theorem.}
\begin{equation}\label{eqn:Dvoretzky-alg}
x_{t+1}(\omega) = T_t\bigl(x_0(\omega),\ldots,x_t(\omega)\bigr) + y_t(\omega),
\end{equation}
where the transformation $T_t$ satisfy, for any $x_0,\ldots,x_t\in\R^n$,
\begin{equation}\label{eqn:Dvoretzky-ineq}
\bigl\|T_t(x_0,\ldots,x_t)-x_*\|^2 \leq \max\bigl\{a_t,\,(1+b_t)\|x_t-x_*\|^2 - c_t + h_t\bigr\} ,
\end{equation}
and the sequences $\{a_t\}$, $\{b_t\}$, $\{c_t\}$ and $\{h_t\}$ are non-negative and satisfy
\begin{align}
\lim_{t\to\infty}a_t = a_\infty, \qquad
%\limsup_{t\to\infty}a_t \leq a_\infty, \qquad
\sum_{t=0}^\infty b_t < \infty, \qquad
\sum_{t=0}^\infty c_t = \infty, \qquad
\sum_{t=1}^\infty h_t < \infty.
\label{eqn:Dvoretzky-conditions}
\end{align}
% \todo{Need to change notation $h_t$ due to clash with search direction. Maybe use $h_t$ here?}
In addition, suppose the following conditions hold with probability one:
\begin{equation}\label{eqn:yt-squared-summable}
\E[\|x_0\|^2]<\infty,\qquad 
\sum_{t=0}^\infty \E[\|y_t\|^2]<\infty, \qquad
\E\bigl[y_t|x_0,\ldots,x_t\bigr] = 0 \quad\forall\,t\geq 0.
\end{equation}
Then we have with probability one,
\[
\limsup_{t\to\infty} \|x_t-x_*\|^2 \leq a_\infty.
\]
\end{theorem}

\paragraph{Remark.} \label{Dvoretzky-remarks}
There are many extensions of the original result by \citet{Dvoretzky1956}. 
Theorem~\ref{thm:Dvoretzky} is a minor variation of \citet[][Theorem~1]{Venter1966}. More concretely,
\begin{itemize}
\item Theorem~1 of \citet{Venter1966} has the sequence $\{a_t\}$ being a constant sequence, i.e., $a_t=a_\infty$ for all $t\geq 0$. The extension to a non-constant sequence $\{a_t\}$ is outlined in the original work of \citet{Dvoretzky1956} and admits a simple proof due to \citet{DermanSacks1959}.
\item Theorem~1 of \citet{Venter1966} does not include the sequence $\{h_t\}$. The extension with $\sum_{t=0}^\infty h_t<\infty$ is straightforward based on a simple argument of \citet{Dvoretzky1956}.
\item More generally, the sequences $\{a_t\}$, $\{b_t\}$, $\{c_t\}$, $\{h_t\}$ can be non-negative measurable functions of $x_0,\ldots,x_t$, and the conclusion of Theorem~\ref{thm:Dvoretzky} holds if $a_\infty$ is an upper bound on $\limsup_{t\to\infty}a_t(x_0,\ldots,x_t)$ uniformly for all sequences $\{x_t\}$ \citep{DermanSacks1959,ROBBINS1971}.
\end{itemize}

\smallskip

\begin{proof}[Proof of Theorem~\ref{thm:Dvoretzky}]
We first write the practical BCOS method in the form of~\eqref{eqn:Dvoretzky-alg}:
\begin{align*}
x_{t+1} 
&= (1-\alpha_t\lambda) x_t - \alpha_t\frac{d_t}{\sqrt{v_t+\epsilon}} \\
&=(1-\alpha_t\lambda)x_t - \alpha_t\E_t\!\left[\frac{d_t}{\sqrt{v_t+\epsilon}}\right] + \alpha_t\left(\E_t\!\left[\frac{d_t}{\sqrt{v_t+\epsilon}}\right]-\frac{d_t}{\sqrt{v_t+\epsilon}}\right).
\end{align*}
Thus we have $x_{t+1}=T_t(x_0,\ldots,x_t)+y_t$ with
\begin{align*}
T_t(x_0,\ldots,x_t)
&=(1-\alpha_t\lambda)x_t - \alpha_t\E_t\!\left[\frac{d_t}{\sqrt{v_t+\epsilon}}\right], \\
y_t&= \alpha_t\left(\E_t\!\left[\frac{d_t}{\sqrt{v_t+\epsilon}}\right]-\frac{d_t}{\sqrt{v_t+\epsilon}}\right).
\end{align*}
Apparently $\E_t[y_t]=\E[y_t|x_0,\ldots,x_t]=0$.
We also have $\sum_{t=0}^\infty\E[\|y_t\|^2]<\infty$ due to the assumptions that $\sum_{t=0}^\infty\alpha_t^2<\infty$ and~$d_t$ is bounded almost surely.
Therefore the conditions in~\eqref{eqn:yt-squared-summable} are all satisfied.
%\todo{Can we show $\E[\|y_t\|^2]$ bounded based on Lemma~\ref{lem:practicalBCOS_error_all}? Then no need to assume $d_t$ bounded.}

The squared distance between $T_t(x_0,\ldots,x_t)$ and $x_*$ admits the following expansion:
\begin{align*}
\bigl\|T_t(x_0,\ldots,x_t)-x_*\bigr\|^2
&=\left\|(1-\alpha_t\lambda)x_t - \alpha_t\E_t\!\left[\frac{d_t}{\sqrt{v_t+\epsilon}}\right] - x_*\right\|^2 \\
&=\left\|(1-\alpha_t\lambda)(x_t-x_*) - \alpha_t\left(\E_t\!\left[\frac{d_t}{\sqrt{v_t+\epsilon}}\right] + \lambda x_*\right)\right\|^2 \\
&=(1-\alpha_t\lambda)^2\|x_t-x_*\|^2 - 2\alpha_t(1-\alpha_t\lambda)\left\langle x_t-x_*,\,\E_t\!\left[\frac{d_t}{\sqrt{v_t+\epsilon}}\right]+\lambda x_*\right\rangle \\
&\quad + \alpha_t^2\left\|\E_t\!\left[\frac{d_t}{\sqrt{v_t+\epsilon}}\right]+\lambda x_*\right\|^2.
\end{align*}
Using Lemma~\ref{lem:practical-aiming} and then the aiming condition~\eqref{eqn:bcosw-aiming-x*}, we obtain
\begin{align*}
\left \langle x_t-x_\ast,\, \E_t\!\left[\frac{d_t}{\sqrt{v_t + \epsilon}}\right]\!+\!\lambda x_*\right \rangle 
&\geq \left \langle x_t-x_\ast,\,\frac{\E_t[d_t]}{\sqrt{\E_t[d_t^2]}}+\!\lambda x_*\right \rangle  - \|x_t-x_*\|_1\left(\sigma_t + \cO(\epsilon)\right) \\
&\geq - \|x_t-x_*\|_1\left(\sigma_t + \cO(\epsilon)\right) .
\end{align*}
By the assumption that $d_t$ is bounded almost surely, there exists a constant $B$ such that 
$$\left\|\E_t\!\left[\frac{d_t}{\sqrt{v_t+\epsilon}}\right]+\lambda x_*\right\|^2\leq B, \qquad \forall\, t\geq 0. $$
Together with $0<1-\alpha_t\lambda<1$, we conclude that
\begin{align}
\bigl\|T_t(x_0,\ldots,x_t)-x_*\bigr\|^2
&\leq (1-\alpha_t\lambda)^2\norm{x_{t}-x_\ast}^2 + 2\alpha_t(1-\alpha_t\lambda)\|x_t-x_*\|_1(\sigma_t+\cO(\epsilon)) + \alpha_t^2 B \nonumber\\
&\leq (1-\alpha_t\lambda)^2\norm{x_{t}-x_\ast}^2 + 2\alpha_t\|x_t-x_*\|_1(\sigma_t+\cO(\epsilon)) + \alpha_t^2 B \nonumber\\
&= (1+\alpha_t^2\lambda^2)\norm{x_{t}-x_\ast}^2 - 2\alpha_t\lambda \|x_t-x_*\|^2 + 2\alpha_t\|x_t-x_*\|_1(\sigma_t+\cO(\epsilon)) \nonumber\\
&\quad~ + \alpha_t^2 B  \nonumber\\
&= (1+\alpha_t^2\lambda^2)\norm{x_{t}-x_\ast}^2 + \alpha_t\left(2\|x_t-x_*\|_1(\sigma_t+\cO(\epsilon))-\lambda\|x_t-x_*\|^2\right)  \nonumber\\
&\quad~ - \alpha_t\lambda \|x_t-x_*\|^2 + \alpha_t^2 B. 
\label{eqn:T-upper-bound}
\end{align}

We observe that there exist $\delta>0$ such that for all $t\geq 0$,
\[
\|x_t-x_*\| \geq \delta \quad \Longrightarrow \quad
2\|x_t-x_*\|_1(\sigma_t+\cO(\epsilon))-\lambda\|x_t-x_*\|^2 \leq 0.
\]
Therefore, when $\|x_t-x_*\|\geq \delta$, we deduce from~\eqref{eqn:T-upper-bound} that 
\[
\norm{T_t(x_0,\ldots,x_t)-x_\ast}^2 
\leq (1+\alpha_t^2\lambda^2)\|x_t-x_*\|^2 - \alpha_t\lambda\|x_t-x_*\|^2 + \alpha_t^2 B.
\]
Otherwise, when $\|x_t-x_*\|\leq\delta$, we have 
\[
\norm{T_t(x_0,\ldots,x_t)-x_\ast}^2 
\leq (1-\alpha_t\lambda)^2\delta^2 + 2\alpha_t (1-\alpha_t\lambda)\sqrt{n}\delta(\sigma_t + \cO(\epsilon)) + \alpha_t^2 B,
\]
where we used $\|x_t-x_*\|_1\leq\sqrt{n}\|x_t-x_*\| \leq \sqrt{n}\delta$.
%Define $a_t$ as the right-hand side of the above inequality, i.e., 
With the definition 
\begin{equation}\label{eqn:a_t-definition}
a_t = (1-\alpha_t\lambda)^2\delta^2 + 2\alpha_t (1-\alpha_t\lambda)\sqrt{n}\delta(\sigma_t + \cO(\epsilon)) + \alpha_t^2 B,
%a_t = (1-\alpha_t\lambda)^2\delta^2 + 2\alpha_t \bigl(c \|\sqrt{\rho_t}\|\delta + \cO(\epsilon)\bigr) + \alpha_t^2 B.
\end{equation}
we can combine the above two cases as
\[
\norm{T_t(x_0,\ldots,x_t)-x_\ast}^2 
\leq \max\bigl\{a_t, \, (1+\alpha_t^2\lambda^2)\|x_t-x_*\|^2 - \alpha_t\lambda\|x_t-x_*\|^2 + \alpha_t^2 B\bigr\}.
\]
With the additional definitions of
\[
b_t=\alpha_t^2\lambda^2, \qquad
c_t=\alpha_t\lambda\|x_t-x_*\|^2, \qquad
h_t=\alpha_t^2 B,
\]
we arrive at the key inequality~\eqref{eqn:Dvoretzky-ineq}. 

In order to apply Theorem~\ref{thm:Dvoretzky}, we are left to check the conditions in~\eqref{eqn:Dvoretzky-conditions}.
Using the assumption $\alpha_t\to 0$, the definition of~$a_t$ in~\eqref{eqn:a_t-definition} implies that $\lim_{t\to\infty}a_t=\delta^2$. The conditions on $\{b_t\}$ and $\{h_t\}$ are satisfied due to the assumption $\sum_{t=0}^\infty\alpha_t^2\leq\infty$.
For $\{c_t\}$, if $\sum_{t=0}^\infty c_t<\infty$, then we must have $\|x_t-x_\star\|^2\to 0$ almost surely because of the assumption~$\sum_{t=0}^\infty\alpha_t=\infty$, 
and the conclusion of the theorem holds trivially.
Otherwise, $\sum_{t=0}^\infty c_t=\infty$ allows all the conditions in~\eqref{eqn:Dvoretzky-conditions} to hold, so we can invoke Theorem~\ref{thm:Dvoretzky} to finish the proof.
\end{proof}

Following the third remark after Theorem~\ref{thm:Dvoretzky}, the radius of the neighborhood~$\delta$ can be refined as follows. 
Let $\delta_t$ be the smallest constant such that $\|x_t-x_*\|\geq \delta_t$ implies
\[
2\|x_t-x_*\|_1\bigl(\sigma_t+\cO(\epsilon)\bigr) \leq \lambda\|x_t-x_*\|^2,
\]
where $\sigma_t$ is given by~\eqref{eqn:c_t-def}.
Then we can set $\delta=\limsup_{t\to\infty}\delta_t$ in~\eqref{eqn:limsup-delta}.

\subsection{Bias and variance trade-off}
\label{sec:bias-var-tradeoff}

Theorem~\ref{thm:bcosw-almost-sure} states that the sequence $\{x_t\}$ generated by the practical BCOS method converges almost surely to a neighborhood of~$x_*$, and the radius of the neighborhood is determined by the magnitude of~$\sigma_t$ and~$\epsilon$. 
Specifically, smaller~$\sigma_t$ and~$\epsilon$ ensure a smaller neighborhood of convergence.
According to~\eqref{eqn:c_t-def}, the magnitude of~$\sigma_t$ depends on both the bias and variance of the estimator~$v_t$. 
In this section, we examine the delicacy of this bias-variance trade-off and illustrate its effects through several examples. 

For convenience, we repeat the definition of~$\sigma_t$ here for quick reference:
\begin{equation}\label{eqn:c_t}
\sigma_t=\frac{\tau \!+\! \cO(\tau^2)}{2} + \left(1+\frac{\tau \!+\! \cO(\tau^2)}{2}\right)\left\|\frac{-(1/2)\corr_t(d_t, v_t)}{\sqrt{\SNR_t(d_t)\SNR_t(v_t\!+\!\epsilon)}} + \cO\!\left(\frac{1}{\SNR_t(v_t\!+\!\epsilon)}\right) \right\|_\infty \!\!.
\end{equation}
Apparently, in order to make~$\sigma_t$ small, it is desirable for the random variable~$v_t$ (as an estimator of $\E_t[d_t^2]$) to possess the following properties.
\begin{itemize}\itemsep 0pt
\item \emph{Low bias}. This means that Assumption~\ref{assum:v_t-bias} holds with small~$\tau$ and~$\epsilon$.
\item \emph{Low variance}, which naturally leads to \emph{high SNR}.
Notice that using a larger~$\epsilon$ increases $\SNR_t(v_t+\epsilon)=\E_t[v_t+\epsilon]^2/\Var_t(v_t)$, leading to a smaller~$\sigma_t$. However, it also increases the $\cO(\epsilon)$ term in~\eqref{eqn:practical-aiming}, which affects the radius of neighborhood through~\eqref{eqn:neighberhood-cond}.
\end{itemize}
In practice, it is hard to construct estimators with low bias and low variance simultaneously.
Below we give several examples to illustrate the trade-off in algorithm design. 
\begin{enumerate}
\item 
\textbf{Using $d_t^2$ as the estimator.}
The simplest estimator of $\E_t[d_t^2]$ is $d_t^2$ itself. 
Setting $v_t=d_t^2$ exhibiting low bias and high variance, specifically,
\[
\mathrm{Bias}=\left|\E_t[v_t]-\E_t[d_t^2]\right|=0, \qquad \mathrm{Variance}=\var_t(v_t)=\var_t(d_t^2).
\]
In this case, Assumption~\ref{assum:v_t-bias} is satisfied with $\tau=0$ and $\epsilon=0$. Then we have
\[
x_{t+1} ~=~ x_t - \alpha_t \frac{d_t}{\sqrt{d_t^2 +\epsilon}} ~=~ x_t - \alpha_t \sign(d_t),
\]
which corresponds to the sign-gradient method~\eqref{eqn:sign-gradient} or the sign-momentum method~\eqref{eqn:sign-momentum}, depending on the choice of~$d_t$.
Between them, sign-momentum is preferred because $\Var_t(m_t^2)$ is usually much smaller than~$\Var_t(g_t^2)$.
\item 
\textbf{Using a constant estimator across coordinates.}
Setting $v_t=c\ones_n$ for some constant $c>0$ exhibits high bias and low variance, specifically,
\[
\mathrm{Bias}=\left|\E_t[v_t]-\E_t[d_t^2]\right|=\left|c\ones_n-\E_t[d_t^2]\right|, \qquad \mathrm{Variance}=\var_t(v_t)=0.
\]
In this case, we need large~$\tau$ and~$\epsilon$ for Assumption~\ref{assum:v_t-bias} to hold. Correspondingly
\[
x_{t+1} = x_t - \alpha_t \frac{d_t}{\sqrt{c\ones_n +\epsilon}} = x_t - \alpha_t' d_t,
\]
where $\alpha_t'=\alpha_t/\sqrt{c +\epsilon}$ is a scalar stepsize.
This is equivalent to the classical SGD with or without momentum, depending on $d_t=m_t$ or $d_t=g_t$ respectively.
\item 
\textbf{Using EMA of $d_t^2$ as the estimator.}
EMA is widely adopted as an online estimator to approximate the mean of a random variable. For example, $v_t=\beta_1 v_{t-1} + (1-\beta)d_t^2$ is used in BCOS-g (Algorithm~\ref{alg:bcos-g}) and BCOS-m (Algorithm~\ref{alg:bcos-m}) to approximate $\E_t[d_t^2]$, with $d_t=g_t$ and $d_t=m_t$ respectively. The bias can be expressed as
\begin{align*}
\mathrm{Bias}&=\left|\E_t[v_t]-\E_t[d_t^2]\right|
=\left|\E_t\biggl[\sum_{k=1}^t (1-\beta) \beta^{t-k} d_k^2\biggr]-\E_t[d_t^2]\right|.
\end{align*}
It is possible to derive a sensible bound on the bias with a smoothness assumption on the loss function, but we do not delve into it here. 
On the other hand, we have a simple expression for its variance (derivation given in Appendix~\ref{sec:appendix-bias-var}):
\begin{align*}
    \var_t(v_t) &= (1-\beta)^2 \var_t ( d_t^2),
\end{align*}
which is much smaller than $\var_t(d_t^2)$ obtained by simply using $d_t^2$ as the estimator.
\item 
\textbf{The Adam estimator.} 
Adam uses stochastic momentum as the search direction, i.e., $d_t=m_t=\beta_1 d_{t-1}+(1-\beta_1)g_t$. 
However, it uses the EMA of~$g_t^2$, rather than the EMA of~$m_t^2$ (as in the previous example), to estimate $\E_t[m_t^2]$.
This causes a mismatch from the perspective of BCOS, and it is compensated by using a larger smoothing factor~$\beta_2$ in the estimator~$v_t=\beta_2 v_{t-1}+(1-\beta_2)g_t^2$.
Its bias can be expressed as
\begin{align*}
\mathrm{Bias}
%&=\bigl|\E_t[v_t]-\E_t[d_t^2]\bigr|
=\left|\E_t\biggl[\sum_{k=1}^t (1-\beta_2) \beta_2^{t-k} g_k^2\biggr]-\E_t\left[\biggl(\sum_{k=1}^t (1-\beta_1) \beta_1^{t-k} g_k\biggr)^{\!2}\right]\right|.
\end{align*}
Again, deriving a sensible upper bound on the bias can be quite involved and we do not pursue it here.
As for the variance, we get (see Appendix~\ref{sec:appendix-bias-var})
\begin{align*}
\var_t(v_t) &= (1-\beta_2)^2 \var_t \left( g_t^2 \right).
\end{align*}
If we choose $\beta_2=1-(1-\beta_1)^2$, then $\var_t(v_t)=(1-\beta_1)^4\var_t(g_t^2)$.
Therefore, the Adam estimator exhibits a very low variance.
\item 
\textbf{The conditional estimator of BCOS-c.}
The BCOS-c method (Algorithm~\ref{alg:bcos-c}) also uses $m_t$ as the search direction. Instead of relying on EMA, it uses a conditional estimator for $\E_t[m_t^2]$, i.e., $v_t=\beta^2m_{t-1}^2 + 2 \beta (1-\beta) m_{t-1} m_t + (1-\beta)^2 g_t^2$, where~$\beta$ is the smoothing factor for computing~$m_t$.
The bias and variance can be expressed as:
%This estimator exhibits low bias and low variance properties:
\begin{align*}
\mathrm{Bias_t}(v_t)
%&=\left|\E_t[v_t]-\E_t[d_t^2]\right|
&=2\beta^2(1-\beta) \left|m_{t-1}\odot\bigl(m_{t-1} - \E_t[g_t]\bigr) \right|, \\
\var_t(v_t) &= (1-\beta)^4 \left(4 \beta^2  m_{t-1}^2 \var_t(g_t) + \,\var_t(g_t^2) + 4 \beta m_{t-1}(\E_t[g_t^3] - \E_t[g_t] \E_t[g_t^2])\right).
\end{align*}
Notice that $|m_{t-1}-\E_t[g_t]|$ is the bias of using $m_{t-1}$ to estimate $\E_t[d_t]$, and the bias of~$v_t$ can be much smaller than that.
% Meanwhile, the variance of BCOS-c is the same as Adam when setting~$\beta_2$ in Adam as $\beta_2=1-(1-\beta)^2$. \todo{TJ: added the full BCOS-c.}
For the simple alternative $v_t=(1-(1-\beta)^2)m_{t-1}^2 + (1-\beta)^2 g_t^2$, we have simple expressions for both the bias and the variance:
%This estimator exhibits low bias and low variance properties:
\begin{align*}
\mathrm{Bias_t}(v_t)
%&=\left|\E_t[v_t]-\E_t[d_t^2]\right|
&=2\beta(1-\beta) \left|m_{t-1}\odot\bigl(m_{t-1} - \E_t[g_t]\bigr) \right|, \\
\var_t(v_t) &= (1-\beta)^4 \,\var_t(g_t^2).
\end{align*}
The variance of the simple alternative is the same as Adam when setting~$\beta_2$ in Adam as $\beta_2=1-(1-\beta)^2$.

\end{enumerate}

In addition to the bias-variance trade-off illustrated above, the expression of~$\sigma_t$ in~\eqref{eqn:c_t} also reveals the following tips for algorithm design:
\begin{itemize}
\item \emph{High SNR of $d_t$}. This is consistent with the empirical observation that using the momentum~$m_t$ as search direction is often better than using the stochastic gradient~$g_t$, because $m_t$ has much smaller variance, thus higher SNR, than~$g_t$.
\item \emph{Positive correlation between~$d_t$ and~$v_t$}. The negative sign before $\corr_t(d_t,v_t)$ implies that a positive correlation can lead to a smaller~$\sigma_t$.
Intuitively, a negative correlation between~$d_t$ and $v_t$ may cause large fluctuations of the ratio $d_t/\sqrt{v_t+\epsilon}$.
\end{itemize}
In summary, our theory captures several essential trade-offs that match intuitions and empirical observations.

\section{Conclusion}
\label{sec:conclusion}

BCOS is a family of stochastic approximation methods that exploit the flexibility of using different block-coordinate stepsizes.
Rather than using sophisticated techniques from optimization such as preconditioning, it is derived from the simple idea of minimizing the distance from the next iterate to a target point.
While the optimal stepsizes are not computable, we make several simplifications and focus on constructing efficient statistical estimators (for the second moment of the search direction) in determining the stepsizes.
In particular, by leveraging a simple conditional estimator, we derive variants of BCOS that achieve competitive performance against AdamW but require less memory and fewer hyperparameters. 

Our convergence analysis builds upon several classical results from the stochastic approximation literature. 
In particular, we extend the classical aiming condition towards a target point to account for coordinate-wise stepsizes and decoupled weight-decay.
%(interpreted as $L_2$-regularization). <-- This is not true!
%This condition does not assume convexity or smoothness, thus has broad applicability.
For the conceptual BCOS method, which assumes exact second moment, we establish almost-sure convergence and $\cO(1/t)$ rate of convergence. 
For practical BCOS methods with various second-moment estimators, we establish almost-sure convergence to a neighborhood of the target point, where the radius of the neighborhood is determined by the bias and variance of the second-moment estimator. 
This framework covers convergence analysis for a broad family of methods, including the 
%stochastic gradient/momentum method, 
%\todo{Added SGD to the algorithm class that we can analyze. But these are too obvious and covered by classical analysis anyway.}
sign stochastic gradient/momentum method, RMSProp and Adam(W).
In addition, it provides novel insights that can guide the development of new algorithms.

\iffalse
Future work: (Too obvious or too secretive, so we skip ...)
\begin{itemize}
\item more empirical study?
stepsize scheduler: currently based on the popular cosine scheduler which works well for Adam. More room for improvement, especially consdering adaptive schedulers!
\item
More work to do on spectral BCOS methods!
\end{itemize}
\fi

%%%%%%%%%%%%%%%%%%%%%%%%%%%%%%%%%%%%%%%%%%%%%%%%%%%%%%%%%%%%
\section*{Acknowledgments}

We thank Zeyuan Allen-Zhu for helping us with running experiments on GPT2, especially with rotary embedding and distributed data parallelism in PyTorch. We are grateful to Lisa Jin for her codebase that made our lives much easier on experimenting with ResNet and vision Transformers.
We also thank Yann Olivier for careful reading of an early version of the draft and providing feedback.
Last but not least, we extend our gratitude to Damek Davis, Anil Damle, Adrian Lewis, James Renegar and Katya Scheinberg for their comments and insights that helped us clarify and refine the theoretical underpinnings of the paper.

%%%%%%%%%%%%%%%%%%%%%%%%%%%%%%%%%%%%%%%%%%%%%%%%%%%%%%%%%%%%
\appendix

\section{}

\subsection{Aiming versus convexity}
\label{sec:aiming-vs-convexity}

Our aiming condition in Assumption~\ref{assum:aiming} does not have an upper bound as in the classical condition~\eqref{eqn:aiming-sakrison}, meaning that it does not require smoothness. 
Smoothness would give a better bound on the bias and variance of the second-moment estimator of the practical BCOS algorithm, but does not affect stability of the algorithm as long as Assumption~\ref{assum:v_t-bias} holds.
Here we focus on the connection of the aiming condition with convexity.

For ease of comparison, we consider the general block structure described in Section~\ref{sec:notations} with $d_{t}=\nabla f(x_t,\xi_t)$ and set $\lambda=0$.
In this case, we have $\E[d_t]=\nabla F(x_t)$ and 
\begin{equation}\label{eqn:general_block_weak_aiming}
    \sum_{k=1}^m \left\langle x_{k}-x_{\ast, k}, ~\frac{\nabla_k F(x)}{\sqrt{\E\bigl[\norm{\nabla_k f(x, \xi)}^2\bigr]}}\right\rangle \geq 0, \qquad \forall\, x,
\end{equation}
where $\nabla_k F(x)\in\R^{n_k}$ denotes the sub-vector of $\nabla F(x)$ corresponding to the $k$th block.
In the specific case of a full-dimensional block ($m=1$), the denominator $\sqrt{\E\bigl[\|\nabla f(x,\xi)\|^2\bigr]}$ can be omitted without affecting the inequality and the aiming condition simplifies to:
\begin{equation}\label{eqn:aiming-full-block}
\left\langle x-x_\ast, \nabla F(x)\right\rangle \geq 0, \qquad \forall\, x.
\end{equation}
This is a direct consequence of convexity when $x_\ast$ is a minimizer of $F$, which is characterized by
\citep[e.g.,][]{Nesterov04book}:
\[
    \left\langle x-y, \nabla F(x) - \nabla F(y)\right\rangle \geq 0, \qquad \forall\, x, y.
\]
To see the implication, simply substituting $y=x_\ast$ and $\nabla F(x_\ast)=0$ into the above inequality gives the full-dimensional aiming condition.

It is easy to come up with examples that satisfy the aiming condition but not convexity.
%To illustrate the aiming condition~\eqref{eqn:aiming-sakrison}, 
For simplicity, we consider the 1-dimensional case and without loss of generality let $x^*=0$. Let $g(x):=\E_{\xi}[g(x,\xi)]$ be any continuous function that satisfies $ g(x) > 0$ for $x> 0$ and $g(x)<0$ for $x<0$, which leads to
$$\langle x-x^*, g(x)\rangle =x \cdot g(x) > 0,$$ for all $x\neq 0$.
For example, consider
$$g(x)=x(x-3)^2(x+2)^2+\lambda x,$$ 
which is a non-monotone polynomial, thus the gradient of a nonconvex polynomial. Yet, this function satisfies the aiming condition~\eqref{eqn:aiming-sakrison} for $x^*=0$ and for any $\lambda\geq 0$. 

%Suppose $g(x)$ is the gradient of some loss function $f(x)$. Then this condition is much weaker than convexity, which would require monotonicity of $g(x)$, that is, $(x-y)(g(x)-g(y))>0$ for all $x,y\in\R$. 

However, other than the full-dimensional block case, the aiming condition~\eqref{eqn:general_block_weak_aiming} is not necessarily implied by convexity.
As an example, consider the case of single coordinate blocks with uniform SiF, i.e., $\rho_{t,k}=\rho_{t,k'}$ for all $k,k'\in\{1,\ldots,n\}$.
In light of the equivalent formulation using SiF in~\eqref{eqn:aiming-sif}, the corresponding aiming condition becomes
\begin{equation}\label{eqn:sign-aiming}
    \left\langle x-x_\ast,\, \sign(\nabla F(x))\right\rangle \geq 0, \qquad \forall\, x.
\end{equation}
Comparing this condition with convexity, it is easy to construct examples that satisfy one of them while failing the other.

To illustrate the fundamental differences between the coordinate-wise aiming condition~\eqref{eqn:sign-aiming} and the standard convexity condition~\eqref{eqn:aiming-full-block}, we provide the following two counterexamples, each satisfying one condition while failing the other:
\begin{itemize}
    \item \textit{Aiming but not convex}: Let $f(x):=\log(x)$ with the target solution $x_\ast = 0$. On the domain of $\R_+$, the gradient is $f'(x)=\frac{1}{x}$, and thus $\sign(f'(x))=1$ for all $x > 0$. Consequently, for any $x\in \R_+$, we have 
    \[
    \langle x-x_\ast, \sign(\nabla f(x))\rangle = x \geq 0,
    \]
    satisfying the aiming condition~\eqref{eqn:sign-aiming}. However, $\log(x)$ is a concave function, thus failing the convex inequality~\eqref{eqn:aiming-full-block}.
    \item \textit{Convex but not aiming}: Consider the quadratic function $f:\R^2 \to \R, f(x) = \frac{1}{2}x^T A x$. Choose the coefficient matrix as
    $$A = \begin{pmatrix}
    1\\
    -2
    \end{pmatrix}
    \begin{pmatrix}
    1\\
    -2
    \end{pmatrix}^T=
    \begin{bmatrix}
        1 & -2\\
        -2 & 4
    \end{bmatrix}\succeq 0.$$
    Since $A$ is positive semidefinite, the function $f$ is convex and attains its minimum at $x_\ast = \mathbf{0}$. The gradient of $f$ is 
    \[
    \nabla f(x) = Ax = \begin{bmatrix}
    x_1 - 2x_2\\
    -2x_1 + 4x_2.
    \end{bmatrix}.
    \]
    Evaluating the aiming condition~\eqref{eqn:sign-aiming} at $x=(1.5,1)^T$, we get 
    \[
    \langle x-x_\ast, \sign(\nabla f(x))\rangle  = 1.5 \times \sign(-0.5) + 1 \times \sign(1)  = 1.5 \times (-1) + 1 \times 1 = -0.5 \leq 0.
    \]
    Thus, the aiming condition~\eqref{eqn:sign-aiming} does not hold at this point even though $f$ is convex. 
\end{itemize}

\subsection{Proof of Lemma~\ref{lem:bound-frac-expect}}
\label{sec:appendix-frac-expect}
\begin{proof}
The proof leverages the Taylor expansion of $\phi(y):=1/\sqrt{y}$. Specifically, we have 
\[
\phi(y+\delta)=\sum_{p=0}^{\infty}\frac{\phi^{(p)}(y)}{p!}\delta^p, 
\]
where
\[
\phi^{(p)}=(-1)^p \left(\frac{1}{2} \cdot \frac{3}{2} \cdots \left(p-\frac{1}{2}\right)\right) y^{-\frac{2p+1}{2}} 
= (-1)^p \frac{(2p)!}{4^p p!} y^{-\frac{2p+1}{2}}.
\]
Taylor expansion (element-wise) at $y=\E_t[d_t^2]$ with $\delta=\E_t[v_t]+\epsilon-\E_t[d_t^2]$ yields the following approximation:
\begin{align}
\left| \frac{\E_t[d_t]}{\sqrt{\E_t[d_t^2]}} - \frac{\E_t[d_t]}{\sqrt{\E_t[v_t]+\epsilon}}\right| 
&= \left| \E_t[d_t](\phi(y)-\phi(y+\delta))\right| 
 =\left| \E_t[d_t]\left(\sum_{p=1}^{\infty}\frac{\phi^{(p)}(y)}{p!}\delta^p\right)\right| \nonumber\\
&= \left|\E_t[d_t]\left(\sum_{p=1}^{\infty}(-1)^p \frac{(2p)!}{4^p (p!)^2} \E_t[d_t^2]^{-\frac{2p+1}{2}} (\E_t[v_t]+\epsilon-\E_t[d_t^2])^p \right)\right|\nonumber\\
&\leq \left|\E_t[d_t]\left(\sum_{p=1}^{\infty} \frac{(2p)!}{4^p (p!)^2} \E_t[d_t^2]^{-\frac{2p+1}{2}} (\tau \E_t[d_t^2] + 2\epsilon)^p \right)\right|\label{eqn:apply-bound-frac-expect}\\
&= \left|\E_t[d_t]\left(\sum_{p=1}^{\infty} \frac{(2p)!}{4^p (p!)^2} \E_t[d_t^2]^{-\frac{2p+1}{2}} \tau^p \E_t[d_t^2]^p + \cO(\epsilon) \right)\right|\nonumber\\
&= \left|\frac{\E_t[d_t]}{\sqrt{\E_t[d_t^2]}}\left(\sum_{p=1}^{\infty} \frac{(2p)! \tau^p}{4^p (p!)^2}  + \cO(\epsilon) \right)\right|\nonumber\\
&= \frac{\tau +\cO(\tau^2)}{2} \left|\frac{\E_t[d_t]} {\sqrt{\E_t[d_t^2]}}\right| + \cO(\epsilon) \nonumber
\end{align}
where the inequality~\eqref{eqn:apply-bound-frac-expect} is a consequence of Assumption~\ref{assum:v_t-bias}.
\end{proof}

\subsection{Proof of Lemma~\ref{lem:expectation_ratio_sqrt}}
\label{sec:proof-lemma-ratio-sqrt}

We first derive a useful approximation for a general smooth function.

\begin{lemma}\label{lem:expectation_g_genearl}
Let $\phi:\R^n\to\R$ be a smooth ($C^\infty$) function and $X\in\R^n$ a random variable. Then we have
\begin{equation}\label{eq:lem:expectation_g_genearl}
    \E[\phi(X)] = \phi(\E[X]) + \frac{1}{2} \bigl\langle \nabla^2 \phi(\E[X]), \cov(X)\bigr\rangle + \sum_{|\bm{p}|=3}^\infty \!\frac{D^{(\bm{p})} \phi(\E[X])}{\bm{p}!} \E\bigl[(X\!-\!\E[X])^{(\bm{p})}\bigr],
\end{equation}
where $\langle\cdot,\cdot\rangle$ denotes matrix inner product, i.e,  $\langle A, B\rangle=\Tr(A^T B)$. 
For the multi-variable derivatives, 
we denote $\bm{p}\in \{0,1,2,\ldots\}^n$, $|\bm{p}|=p_1+\cdots+p_n$,
$\bm{p}!=p_1!\cdots p_n!$, and
\begin{align*}
D^{(\bm{p})}\phi 
&=\frac{\partial^{|p|} \phi}{\partial X^p} = \frac{\partial^{p_1+\cdots+p_n} \phi}{\partial X_1^{p_1} \cdots \partial X_n^{p_n}}, \\
(X-\E[X])^{(\bm{p})} &= (X_1-\E[X_1])^{p_1}\cdots (X_n-\E[X_n])^{p_n}.
\end{align*}
\end{lemma}
   
\begin{proof}
    Let $\delta:=X-\E[X]$ and consider the Taylor expansion of~$\phi$ at $\E[X]$:
\begin{align*}
    \phi(X) &= \phi(\E[X]) + \nabla \phi(\E[X])^T \delta + \frac{1}{2}\delta^T \nabla^2 \phi(\E[X]) \delta + \sum_{|\bm{p}|=3}^\infty \frac{D^{(\bm{p})} \phi(\E[X])}{\bm{p}!} \delta^{(\bm{p})}.
\end{align*}
Taking expectation with respect to $X$, we have $\E[\delta]=\E\bigl[X-\E[X]\bigr]=0$. Therefore,
\[
\E\bigl[\nabla \phi(\E[X])^T \delta\bigr] = \nabla \phi(\E[X])^T \E\bigl[\delta\bigr] = 0,
\]
which, together with the equation $\cov(X)=\E[\delta\delta^T]$, yields the desired result.
\end{proof}

\smallskip

\begin{proof}[Proof of Lemma~\ref{lem:expectation_ratio_sqrt}]
    We apply Lemma~\ref{lem:expectation_g_genearl} with $X:=(Y,Z)$ and $\phi(x)=\phi(y,z):= \frac{y}{\sqrt{z}}$. First, the gradient and Hessian of~$g$ can be calculated as
\[
    \nabla \phi(x)=\nabla \phi(y,z) 
    %= \begin{pmatrix}
    = \begin{pmatrix} \displaystyle
        \frac{1}{z^{1/2}}\\[2ex]
        \displaystyle
        -\frac{y}{2z^{3/2}}
    \end{pmatrix},\qquad 
    \nabla^2 \phi(x) = \nabla^2 \phi(y,z) = \begin{bmatrix}
    \displaystyle 0, & \displaystyle -\frac{1}{2z^{3/2}}\\[2ex]
    \displaystyle -\frac{1}{2z^{3/2}}, &\displaystyle  \frac{3y}{4z^{5/2}}
    \end{bmatrix}.
\]
    For general $p$-th partial derivative, we derive the following result for any $q\in\{0,\ldots,p\}$:
\begin{align*}
        \frac{\partial^p \phi}{\partial y^q \partial z^{p-q}}= \frac{\partial^{p-q}}{\partial z^{p-q}}\left(\frac{\partial^q \phi}{\partial y^q}\right) = \begin{cases}
            \displaystyle
            0 \quad &\text{if } q \geq 2,\\[2ex]
            \displaystyle
            \frac{\partial^{p-1} }{\partial z^{p-1}} \frac{1}{\sqrt{z}}= (-1)^{p-1} \frac{(2p-2)!}{4^{p-1} (p-1)!} z^{-\frac{2p-1}{2}} \quad &\text{if } q = 1,\\[2ex]
            \displaystyle
            y\cdot\frac{\partial^{p} }{\partial z^{p}} \frac{1}{\sqrt{z}}= (-1)^{p} \frac{(2p)!}{4^{p} p!} yz^{-\frac{2p+1}{2}} \quad &\text{if } q = 0.\\
        \end{cases}
\end{align*}
Substitute the Hessian and $p$-th order partial derivative into~\eqref{eq:lem:expectation_g_genearl}, we get
\allowdisplaybreaks
\begin{align*}
\E\left[\frac{Y}{\sqrt{Z}}\right]
        &=\frac{\E[Y]}{\sqrt{\E[Z]}} - \E\left[\frac{(Y-\E[Y])(Z-\E[Z])}{2\E[Z]^{3/2}}\right] + \E\left[\frac{3\E[Y](Z-\E[Z])^2}{8\E[Z]^{5/2}}\right]\\
        &\quad +\sum_{p=3}^\infty \frac{1}{p!}(-1)^{p-1}\frac{p(2p-2)!}{4^{p-1} (p-1)!} \E\left[\frac{(Y-\E[Y])(Z-\E[Z])^{p-1}}{\E[Z]^{\frac{2p-1}{2}}} \right]\\
        &\quad + \sum_{p=3}^\infty \frac{1}{p!} (-1)^{p} \frac{(2p)!}{4^{p} p!} \E\left[\frac{\E[Y](Z-\E[Z])^p}{\E[Z]^{\frac{2p+1}{2}}}\right]\\
        &=\frac{\E[Y]}{\sqrt{\E[Z]}} -\frac{\cov(Y,Z)}{2\E[Z]^{3/2}} + \frac{3\E[Y]\var(Z)}{8\E[Z]^{5/2}} \\
        &\quad + \cO\left(\frac{\E\bigl[(Y-\E[Y])(Z-\E[Z])^{p-1}\bigr]}{\E[Z]^{p-1/2}}\right) + \cO\left(\frac{\E[Y]\E\bigl[(Z-\E[Z])^p\bigr]}{\E[Z]^{p+1/2}}\right)\\
        &=\frac{\E[Y]}{\sqrt{\E[Z]}}\Biggl(1 -\frac{\cov(Y,Z)}{2\E[Y]\E[Z]} + \frac{3\var(Z)}{8\E[Z]^2} \\
        &\quad + \cO\left(\frac{\E\bigl[(Y-\E[Y])(Z-\E[Z])^{p-1}\bigr]}{\E[Y]\E[Z]^{p-1}}\right) + \cO\left(\frac{\E\bigl[(Z-\E[Z])^p\bigr]}{\E[Z]^{p}}\right)\Biggr),
\end{align*}
which is the desired result.
\end{proof}

\subsection{Bias and variance calculation}
\label{sec:appendix-bias-var}
Here we derive the bias and variance expressions for some estimators listed in Section~\ref{sec:bias-var-tradeoff}.
\begin{itemize}
\item \textbf{Using EMA of $d_t^2$ as estimator.}
In this case, we have $v_t=\beta_1 v_{t-1} + (1-\beta)d_t^2$ and
    \begin{align*}
        \var_t(v_t) &= \E_t \!\left[\left(\sum_{k=1}^t (1-\beta) \beta^{t-k} d_k^2 - \sum_{k=1}^t (1-\beta) \beta^{t-k} \E_t[d_k^2]\right)^2\right]\\
        &= \E_t \!\left[\left(\sum_{k=1}^{t-1} (1-\beta) \beta^{t-k} d_k^2+ (1-\beta) d_t^2 - \sum_{k=1}^{t-1} (1-\beta) \beta^{t-k} d_k^2 - (1-\beta) \E_t[d_k^2]\right)^2\right]\\
        &= (1-\beta)^2 \E_t \!\left[\left( d_t^2 - \E_t[d_t^2]\right)^2\right]\\
        &= (1-\beta)^2 \var_t \left( d_t^2\right).
    \end{align*}
 
\item \textbf{The Adam estimator.}
In this case, we have $v_t=\beta_2 v_{t-1}+(1-\beta_2)g_t^2$ and
    \begin{align*}
        \var_t(v_t) &= \E_t \left[\left(\sum_{k=1}^t (1-\beta_2) \beta_2^{t-k}  g_k^2 -\sum_{k=1}^t (1-\beta_2) \beta_2^{t-k} \E_t\left[g_k^2\right]\right)^2\right]\\
        &= \E_t \left[\left(\sum_{k=1}^{t-1} (1-\beta_2) \beta_2^{t-k} g_k^2+ (1-\beta_2) g_t^2 - \sum_{k=1}^{t-1} (1-\beta_2) \beta_2^{t-k} g_k^2 - (1-\beta_2) \E_t[g_t^2]\right)^2\right]\\
        &= (1-\beta_2)^2 \E_t \left[\left( g_t^2 - \E_t[g_t^2]\right)^2\right]\\
        &= (1-\beta_2)^2 \var_t \left( g_t^2 \right).
    \end{align*}
\item 
\textbf{The conditional estimator of BCOS-c.}
In this case, we have $d_t=m_t=\beta m_{t-1} + (1-\beta)g_t$ and 
$v_t = \beta^2 m_{t-1}^2 + 2\beta (1-\beta) m_{t-1} m_{t} + (1-\beta)^2 g_t^2$. Therefore, we derive the bias and variance as follows
\begin{align*}
        \mathrm{Bias}&=|\E_t[v_t]-\E_t[d_t^2]|\\
        &=\Big|\E_t\left[\beta^2 m_{t-1}^2 + 2\beta (1-\beta) m_{t-1} (\beta m_{t-1} + (1-\beta)g_t) + (1-\beta)^2 g_t^2\right]\\
        &\quad -\E_t\left[\left(\beta m_{t-1} + (1-\beta)g_t\right)^2\right]\Big|\\
        &=\Big|\beta^2 m_{t-1}^2  + 2\beta (1-\beta)m_{t-1}(\beta m_{t-1} + (1-\beta)\E_t[g_t])
        + (1-\beta)^2 \E_t\left[g_t^2\right]\\
        &\quad -\beta^2 m^2_{t-1} - 2\beta (1-\beta) m_{t-1} \E_t\left[g_t\right] - (1-\beta)^2\E_t\left[g_t^2\right]\Big|\\
        &=\left|2\beta (1-\beta)m_{t-1} (\beta m_{t-1} + (1-\beta)\E_t[g_t] - \E_t[g_t]) \right|\\
        &=2\beta^2(1-\beta) \left|m_{t-1}\left(m_{t-1} - \E_t\left[g_t\right]\right) \right|,
    \end{align*}
and
    \begin{align*}
        \var_t(v_t) &= \E_t \Big[\Big(\beta^2 m_{t-1}^2 + 2\beta (1-\beta) m_{t-1} (\beta m_{t-1} + (1-\beta)g_t) + (1-\beta)^2 g_t^2 \\
        &\quad - \beta^2 m_{t-1}^2 - 2\beta (1-\beta) m_{t-1} (\beta m_{t-1} + (1-\beta)\E_t[g_t]) - (1-\beta)^2 \E_t[g_t^2]\Big)^2\Big]\\
        &= \E_t \Big[\Big(2\beta (1-\beta)^2 m_{t-1} (g_t - \E_t[g_t]) + (1-\beta)^2 (g_t^2 -\E_t[g_t^2])\Big)^2\Big]\\
        &= 4\beta^2 (1-\beta)^4 m_{t-1}^2 \E_t[(g_t-\E_t [g_t])^2] 
        + 4\beta (1-\beta)^4 m_{t-1} \E_t[(g_t - \E_t[g_t])(g_t^2 -\E_t[g_t^2])]\\
        &\quad + (1-\beta)^4 \E_t \left[\left(g_t^2-\E_t[g_t^2]\right)^2\right]\\
        &= 4\beta^2 (1-\beta)^4 m_{t-1}^2 \var_t(g_t) + 4\beta (1-\beta)^4 m_{t-1} \left(\E_t[g_t^3] - \E_t[g_t] \E_t[g_t^2]\right) + (1-\beta)^4 \var_t \left(g_t^2\right).
    \end{align*}

The simple alternative $v_t = (1-(1-\beta)^2)m_{t-1}^2 + (1-\beta)^2 g_t^2$ has bias and variance in simpler form:
    \begin{align*}
        \mathrm{Bias}&=|\E_t[v_t]-\E_t[d_t^2]|\\
        &=\left|\E_t\left[(1-(1-\beta)^2)m_{t-1}^2 + (1-\beta)^2 g_t^2\right]-\E_t\left[\left(\beta m_{t-1} + (1-\beta)g_t\right)^2\right]\right|\\
        &=\left|(2\beta-\beta^2)m_{t-1}^2 + (1-\beta)^2 \E_t\left[g_t^2\right]-\beta^2 m^2_{t-1} - 2\beta (1-\beta) m_{t-1} \E_t\left[g_t\right] - (1-\beta)\E_t\left[g_t^2\right]\right|\\
        &=\left|(2\beta-2\beta^2)m_{t-1}^2 - 2\beta (1-\beta) m_{t-1} \E_t\left[g_t\right] \right|\\
        &=2\beta(1-\beta) \left|m_{t-1}\left(m_{t-1} - \E_t\left[g_t\right]\right) \right|,
    \end{align*}
and
    \begin{align*}
        \var_t(v_t) &= \E_t \left[\left((1-(1-\beta)^2)m_{t-1}^2 + (1-\beta)^2 g_t^2-(1-(1-\beta)^2)m_{t-1}^2 - (1-\beta)^2 \E_t[g_t^2]\right)^2\right]\\
        &= (1-\beta)^4 \E_t \left[\left(g_t^2-\E_t[g_t^2]\right)^2\right]\\
        &= (1-\beta)^4 \var_t \left(g_t^2\right).
    \end{align*}
\end{itemize}

\vskip 0.2in
\bibliography{bcos}

\end{document}